\documentclass[final,leqno]{siamltex}

\usepackage{graphicx}
\usepackage{amsmath}
\usepackage{amsfonts}
\usepackage{amssymb}
\usepackage{mathrsfs}

\usepackage[lofdepth,lotdepth]{subfig}

\usepackage{color}
\definecolor{newblue}{RGB}{94,89,144}
\definecolor{newblue2}{cmyk}{1,0.6,0,0.06}
\usepackage[pdfborder={0 0 0},
  colorlinks=true,
  citecolor=newblue,
  linkcolor=newblue2,
  urlcolor=newblue]{hyperref}
\newcommand{\nautoref}[2]{\hyperref[#2]{#1~\ref{#2}}}

\usepackage{enumerate}

\newtheorem{remark}[theorem]{Remark}

\numberwithin{equation}{section}
\numberwithin{theorem}{section}

\newcommand{\dd}{\,\mathrm{d}}
\newcommand{\abs}[1]{\left|#1\right|}
\newcommand{\norm}[1]{\left\|#1\right\|}
\newcommand{\R}{\mathbb{R}}
\newcommand{\N}{\mathcal{N}}

\newcommand{\T}{\mathscr{T}}

\newcommand{\Hm}{\mathscr{H}}

\newcommand{\md}{\partial^\bullet}

\newcommand{\id}{\mathrm{Id}}

\renewcommand{\tilde}{\widetilde}

\usepackage[notref, notcite, color]{showkeys}

\usepackage{datetime}

\title{Unfitted finite element methods using bulk meshes for surface partial differential equations}
\author{Klaus Deckelnick \footnotemark[1]\ \and
  Charles M.~Elliott \footnotemark[2]\ \and
  Thomas Ranner \footnotemark[2]%
  \thanks{%
   The work of C. M. Elliott was supported by the Warwick
Impact Fund. The work of T. Ranner was supported by
an EPSRC Ph.D. studentship (Grant EP/P504333/1 and EP/P50516X/1) and the Warwick
Impact Fund.   }%
}

\begin{document}

\maketitle

\renewcommand{\thefootnote}{\fnsymbol{footnote}}

\footnotetext[2]{Otto-von-Guericke-Universit\"{a}t Magdeburg,
Institut f\"{u}r Analysis und Numerik,
Universit\"{a}tsplatz 2,
39106 Magdeburg,
Germany (klaus.deckelnick@ovgu.de)}
\footnotetext[3]{Mathematics Institute, Zeeman Building, University of Warwick, Coventry. CV4 7AL. UK (C.M.Elliott@warwick.ac.uk, T.Ranner@warwick.ac.uk)}

\renewcommand{\thefootnote}{\arabic{footnote}}

\begin{abstract}
  In this paper, we define  new unfitted finite element methods for numerically approximating the solution of surface partial differential equations using bulk finite elements. The key idea is that the $n$-dimensional hypersurface, $\Gamma \subset \R^{n+1}$, is embedded in a polyhedral domain in $\mathbb R^{n+1}$ consisting of a union, $\T_h$, of $(n+1)$-simplices. The finite element approximating space is based on continuous piece-wise linear finite element functions on $\T_h$. Our first method is a  sharp interface method, \emph{SIF}, which uses the bulk finite element space in  an approximating weak formulation obtained from integration on a polygonal approximation, $\Gamma_{h}$, of $\Gamma$. The full gradient is used rather than the projected tangential gradient and it is this which distinguishes \emph{SIF} from the method of \cite{OlsReuGra09}. The second method, \emph{NBM}, is a narrow band method in which the region of integration is a narrow band of width $O(h)$. \emph{NBM} is similar to the method of \cite{DecDziEll10} but again the full gradient is used in the discrete weak formulation. The a priori error analysis in this paper shows that the methods are of optimal order in the surface $L^{2}$ and $H^{1}$ norms and have the advantage that the normal derivative of the discrete solution is small and converges to zero. Our third method combines bulk finite elements, discrete sharp interfaces and narrow bands in order to give an unfitted finite element method for parabolic equations on evolving surfaces. We show that our  method is conservative so that it preserves mass in the case of an advection diffusion conservation law. Numerical results are given which illustrate the rates of convergence.
\end{abstract}

\begin{keywords}
  Unfitted finite elements; cut cells; error analysis; narrow band; sharp interface; elliptic and parabolic surface equations
\end{keywords}


\pagestyle{myheadings}
\thispagestyle{plain}
\markboth{K. DECKELNICK, C.~M. ELLIOTT AND T. RANNER}{UNFITTED FINITE ELEMENT METHODS FOR SURFACE PDES}


\section{Introduction}

\subsection{Model equations and motivation}

In this article we propose and analyse numerical methods based on bulk finite element meshes for the following model elliptic equation on a stationary surface.

\noindent \emph{Model elliptic equation on stationary surface}:
Let $\Gamma$ be a smooth hypersurface in $\R^{n+1}$ and $f \in L^2(\Gamma)$. We seek solutions $u\colon \Gamma \to \R$ of
\begin{equation}
  \label{eq:poisson}
  -\Delta_\Gamma u + u = f \quad \mbox{ on } \Gamma.
\end{equation}
The methods can be extended in natural ways to deal with variable coefficients and nonlinearities. The approach may be extended to the following advection diffusion equation on a moving surface.

\noindent \emph{Model parabolic equation on evolving surface}:
Let $\{ \Gamma(t) \}$ be a family of smooth hypersurfaces in $\R^{n+1}$ for $t \in [0,T]$. We seek solutions $u \colon \bigcup_{t} \Gamma(t) \times \{ t \}$ of  the advection diffusion equation
\begin{subequations}
  \label{advdiff}
  \begin{align}
    \md u + u \nabla_\Gamma \cdot v - \Delta_\Gamma u  &= f \quad \mbox{ on } \bigcup_{t \in (0,T)} \Gamma(t) \times \{ t \},
    \\
    u(\cdot, 0)  & =  u_0 \quad \mbox{ on } \Gamma(0).
  \end{align}
\end{subequations}
Here, $\md u$ denotes the material derivative of $u$ and $v$ is the velocity vector. See \S\ref{hybridmoving} for notation.

Surface partial differential equations or  partial differential equations (PDEs) on manifolds  arise in a wide variety of applications in materials science, fluid dynamics and biology, \cite{Sto90, Mul95,  SchDes97, BarMaiPad04, NovGaoCho07, LaiTseHua08, EilEll08, GanTob09, EllSti10siam, BarEllMad11, ErlMcC12, EllStiVen12, RatRog12, MarOrl13}. Often the surface on which the PDE holds is unknown and has to be found as part of the solution process. Thus complex applications involving surfaces and interfaces frequently require the formulation and approximation of equations on unknown stationary and moving surfaces and are not only coupled to equations for the surface but also to equations holding in the bulk. A number of computational approaches have been developed in recent years, see \cite{DziEll13-a}. They are often designed  in the context of solving the surface equations in a more complex application. In particular we mention:-

\noindent {\it Surface finite elements on triangulated surfaces}:
This approach  was pioneered in \cite{Dzi88} for computing solutions to the Poisson equation using piece-wise linear elements on triangulated surfaces. This was extended to nonlinear and fourth order surface parabolic equations in \cite{DziEll07-b}. Using an appropriate weak formulation and the transport property of finite element basis functions an evolving surface finite element method was devised in \cite{DziEll07-a} in order to treat conservation laws on moving surfaces. The key idea is to use the Leibniz (or transport) formula for the time derivative of integrals over moving surfaces in order to derive weak and variational formulations. Further numerical analysis of surface finite element methods may be found in \cite{DemDzi07, Dem09, DziEll07-b, DziEll13, DziEll12, DziLubMan12}.

\noindent{\it Bulk finite element or finite difference meshes for the approximation of implicit surface formulations}:
\label{en:eulerian}
The idea here is to use implicit surface formulations. The starting point is to use a level set function  $\Phi$ to define a degenerate partial differential equation whose solution solves the surface equation on all level sets of $\Phi$. Such methods are formulated in \cite{BerCheOsh01, GreBerSap06, Bur09, DziEll08-a, DziEll10}. Approaches to obtaining a non-degenerate level set equation may be found in \cite{Gre06, CheOls13}.

\noindent {\it Bulk finite element or finite difference meshes on narrow bands}:
\label{en:narrowband}
A natural disadvantage of the approach \ref{en:eulerian} is the fact that formulation is in the ambient space rather than just on the surface. This leads to solving PDEs in one space dimension higher than the hypersurface. Narrow band methods confine the use of bulk finite elements to a narrow band around the surface and the region of integration is the narrow band.  In particular the level  set approach in \cite{DecDziEll10} is confined to an $O(h)$-narrow band. Another direction is to solve the bulk PDE in a narrow band of width $\epsilon$, say. The solution of this problem converges to the solution of the surface PDE. This is the basis of the finite difference method in  \cite{SchAdaCol05}.

\noindent {\it Bulk finite element methods and phase field diffuse interfaces}:
This approach is motivated by  the diffuse interfaces that arise  in phase field approximations of interface problems, see e.g.\ \cite{DecDziEll05}.  The idea is to exploit the methodology to generate methods for solving partial differential equations on the interfaces, \cite{RatVoi06},
 \cite{EllStiSty11}.

\noindent {\it Bulk finite element meshes and sharp interface weak forms}
\label{en:sharpinterface}
If one takes the width of the narrow band to be zero in the approach of \cite{DecDziEll10} one obtains the appealing method of  \cite{OlsReuGra09} for equations on stationary surfaces. The authors  prove that for piece-wise linear elements one obtains second order convergence in the $L^{2}$ norm. There is an issue about the dimension of the resulting linear algebraic equations and their conditioning. This is addressed in \cite{OlsReu10}. For other developments we refer to \cite{DemOls12} which concerns an adaptive version and  \cite{OlsReuXu13} which concerns the surface meshes induced by the bulk mesh on level sets.This approach has been extended recently to a novel Eulerian  space-time formulation using space-time bulk finite element meshes,  \cite{OlsReuXu13-a}.

An important feature of the methods described above is the avoidance of \emph{charts} both in the problem formulation and the numerical methods. For example, the surface finite element method is based simply on triangulated surfaces and requires the geometry solely through the knowledge of the vertices of the triangulation whereas methods based on implicit surfaces require only the level set function $\Phi$ which encodes all the necessary geometry.

Another feature of some of these methods is the use of unfitted bulk meshes. Here we use the terminology \emph{unfitted finite element methods} (sometimes called cut cell methods) when the underlying meshes that form the computational domain are not fitted to the domain in which the PDE holds. The motivation for using finite element spaces on meshes not fitting to the domain came from the desire to solve free or moving boundary problems. Such methods were introduced in \cite{BarEll82, BarEll84}  for elliptic equations in curved domains. See also \cite{HanHan02, BasEng09, EngHei12}. In this setting we are concerned with bulk meshes independent of the surface.

\subsection{The new methods}

The new unfitted finite element  methods for surface elliptic equations  proposed in this paper are  variants of the bulk finite element  approaches using a sharp interface or a narrow band. The new scheme for advection diffusion on an evolving surface is a hybrid of these. In the following we sketch the main ideas of these methods describing the details in \S\ref{SIF}-\S\ref{hybridmoving}.

\subsubsection{Sharp interface method (\emph{SIF})}

Given an interpolation $\Gamma_{h}$ of $\Gamma$, we use a bulk finite element space $V^I_h$ of the form
\begin{equation*}
  V_h^{I} = \{ \phi_h \in C^0(U_h^I)  \, | \, \phi_{h|T} \in P_1(T) \mbox{ for each }  T \in \T_h^I \},
\end{equation*}
where $\T_{h}^{I}$ is a set of elements which intersect $\Gamma_{h}$ and $U_h^I=\bigcup_{T \in \T_h^I} T$, see \S\ref{SIF}. The discrete scheme approximating the model elliptic equation \eqref{eq:poisson} is:- Find $u_h \in V_h^{I}$ such that
\begin{equation}
  \label{sim}
  \int_{\Gamma_h} \bigl( \nabla u_h \cdot \nabla \phi_h + u_h \phi_h \bigr) \dd \sigma_h = \int_{\Gamma_h} f^e \phi_h \dd \sigma_h
  \quad \mbox{ for all } \phi_h \in V_h^{I},
\end{equation}
where $f^{e}$ is an extension of $f$.

The method is related to  the following method of Olshanskii et. al., introduced in  \cite{OlsReuGra09}:-  Find $u_h \in V^{\Gamma}_h$ such that
\begin{equation*}
  \int_{\Gamma_h} \bigl( \nabla_{\Gamma_h} u_h \cdot \nabla_{\Gamma_h} \phi_h + u_h \phi_h \bigr) \dd \sigma_h = \int_{\Gamma_h} f^e \phi_h \dd \sigma_h
  \quad \mbox{ for all } \phi_h \in V^{\Gamma}_h.
\end{equation*}
Apart from the use of the full gradient in (\ref{sim}) as opposed to the tangential gradient, another difference relates to the use of the finite element space  $V^{\Gamma}_h$,  which essentially consists of the traces on $\Gamma_h$ of elements in $V_h^{I}$. However, while $V_h^{I}$ has a natural basis, this does not seem to be the case for $V^{\Gamma}_h$.

\subsubsection{Narrow band method (\emph{NBM})}

We use  the bulk finite element space $V_h^{B}$ on the triangulation $\T^B_h$
\begin{equation*}
  V_h^{B} = \{ \phi_h \in C^0(U^B_h) \, | \, \phi_{h|T} \in P_1(T) \mbox{  for each } T \in \T^B_h \}.
\end{equation*}
Here $\T^{B}_{h}$ consists of those triangles intersecting a narrow band domain $D_{h}$ defined by the $\pm h$ level sets of an interpolated level set function $I_h \Phi$ and $U_h^B=\bigcup_{T \in \T_h^B} T$. The discrete scheme approximating the model elliptic equation (\ref{eq:poisson}) is:- Find $u_h \in V_h^{B}$ such that
\begin{equation}
  \label{nbm}
  \int_{D_h} \bigl( \nabla u_h \cdot \nabla \phi_h + u_h \phi_h \bigr) \abs{ \nabla I_h \Phi }  \dd x
  = \int_{D_h} f^e \phi_h \abs{ \nabla I_h \Phi } \dd x
  \quad \mbox{ for all } \phi_h \in V_h^B.
\end{equation}

This is similar to the method in \cite{DecDziEll10} except that \emph{NBM} uses the full instead of projected gradients thus avoiding the resulting degeneracy. As a result we are able to prove an optimal $L^2$-error bound which was not obtained for the method in \cite{DecDziEll10}. It is also the case that the normal derivative of the discrete solution converges to zero.

\subsubsection{Hybrid unfitted evolving surface method}

The discrete problem approximating (\ref{advdiff}) is:- Given  $u_h^m \in V^{m}_h, m=0,\ldots,N-1$, find $u^{m+1}_h \in V^{m+1}_h$ such that
\begin{equation}
  \begin{aligned}
    & \int_{\Gamma_h^{m+1}} u_h^{m+1} \phi_h \dd \sigma_h - \int_{\Gamma_h^{m}} u_h^m \phi_h  (\cdot + \tau_m v^{e,m+1}) \dd \sigma_h \\
    & + \frac{\tau_m}{2 h} \int_{D_h^{m+1}} \nabla u_h^{m+1} \cdot \nabla \phi_h \, \abs{ \nabla I_h \Phi^{m+1} }  \dd x
    =   \tau_m  \int_{\Gamma_h^{m+1}} f^{e,m+1} \phi_h \dd \sigma_h
  \end{aligned}
\end{equation}
for all $\phi_h \in V^{m+1}_h$.  Here $v^{e,m}$ denotes an extension of the surface velocity at time level $m$. We use  time step labelled analogues of the notation for the narrow band method, see \S\ref{hybridmoving} for the details. Here, $u^0_h $ is appropriate initial data.  An important property of solutions of \eqref{advdiff} is conservation of mass in the case that $f \equiv 0$ and our numerical scheme preserves this property under some mild constraints on the discretization parameters, see \S\ref{hybridmoving}.

\subsection{Outline}

The paper is organized as follows: in \S\ref{prelim} we introduce our notation and collect some auxiliary results. In \S\ref{SIF} and \S\ref{NBM} we present and analyse unfitted methods for the model elliptic equation \eqref{eq:poisson}. In \S\ref{hybridmoving} we describe how a combination of these two approaches can be used to calculate solutions of the advection-diffusion equation on evolving hypersurfaces, \eqref{advdiff}. Details of the implementation and several numerical examples illustrating the orders of convergence are presented in \S\ref{numerical}.


\section{Preliminaries}
\label{prelim}

\subsection{Surface calculus}
\label{surfcalc}

Let $\Gamma$ be a connected compact smooth hypersurface embedded in $\R^{n+1}$ $(n=1,2)$. We assume that there exists a smooth function $\Phi:U \rightarrow \mathbb{R}$ such that
\begin{equation*}
  \Gamma = \lbrace x \in U \, | \, \Phi(x)=0 \rbrace \quad \mbox{ and } \quad \nabla \Phi(x) \neq 0, x \in U,
\end{equation*}
where $U$ is an open neighbourhood of $\Gamma$. For a function $z:\Gamma \rightarrow \R$ we define its tangential gradient by
\begin{equation}
  \label{tanggrad}
  \nabla_{\Gamma} z(p):= \nabla \tilde{z}(p) - \bigl( \nabla \tilde{z}(p) \cdot \nu(p)  \bigr) \nu(p), \quad p \in \Gamma,
\end{equation}
where $\tilde{z}:U \rightarrow \R$ is an arbitrary extension of $z$ to $U$ and
\begin{equation*}
  \nu(x) = \frac{\nabla \Phi(x)}{\abs{ \nabla \Phi(x) }}
\end{equation*}
is a unit vector to the level sets of $\Phi$. It can be shown that $\nabla_{\Gamma} z(p)$ is independent of the particular choice of $\tilde{z}$. We denote by $\underline{D}_i z, 1 \leq i \leq n+1$ the components of $\nabla_{\Gamma} z$. Furthermore, we let
\begin{equation*}
  \Delta_\Gamma z = \nabla_\Gamma \cdot \nabla_\Gamma z = \sum_{i=1}^{n+1} \underline{D}_i \underline{D}_i z
\end{equation*}
be the Laplace-Beltrami operator of $z$.

In what follows it will be convenient to use special coordinates which are adapted to $\Phi$. Consider for $p \in \Gamma$ the system of ODEs
\begin{equation}
  \label{odesys}
  \gamma_p'(s) = \frac{\nabla \Phi(\gamma_p(s))}{| \nabla \Phi(\gamma_p(s)) |^2}, \quad \gamma_p(0)=p.
\end{equation}
It can be shown that there exists $\delta>0$ so that the solution $\gamma_p$ of \eqref{odesys} exists uniquely on $(-\delta,\delta)$ uniformly in $p \in \Gamma$, so that we can define the mapping $F:\Gamma \times (-\delta,\delta) \rightarrow \mathbb{R}^{n+1}$ by
\begin{equation}
  \label{defF}
  F(p,s):=\gamma_p(s), \quad p \in \Gamma, |s| < \delta.
\end{equation}
Since $\frac{d}{ds} \Phi(\gamma_p(s)) = 1$ and $\gamma_p(0)=p \in \Gamma$, we infer that $\Phi(\gamma_p(s))=s, |s| < \delta$ and hence that $x=F(p,s)$ implies that $| \Phi(x) | < \delta$. As a result, we deduce that $F$ is a diffeomorphism of $\Gamma \times (-\delta,\delta)$ onto $U_{\delta}:= \lbrace x \in U \, | \,
| \Phi(x) | < \delta \rbrace$ with inverse
\begin{equation}
  \label{eq:Finv}
  F^{-1}(x)= (p(x),\Phi(x)), \quad x \in U_{\delta},
\end{equation}
where $p: U_{\delta} \rightarrow \mathbb{R}^{n+1}$ satisfies $p(x) \in \Gamma, x \in U_{\delta}$. For later purposes it is convenient to expand $p$ and its derivatives in terms of $\Phi$. Let us fix $x \in U_{\delta}$ and define the function
\begin{equation*}
  \eta(\tau) := F(p(x), (1-\tau) \Phi(x)), \tau \in [0,1].
\end{equation*}
Since $\frac{\partial F}{\partial s}(p,s) = \gamma_p'(s)$ we have
\begin{equation*}
  \eta'(\tau)
  = - \Phi(x) \gamma_{p(x)}'((1-\tau) \Phi(x))
  = -\Phi(x) \frac{ \nabla \Phi( \gamma_{p(x)} ((1-\tau) \Phi(x)))}{ \abs{ \nabla \Phi( \gamma_{p(x)} ( ( 1 - \tau ) \Phi(x))) }^2 }.
\end{equation*}
Observing that $\gamma_{p(x)} ( \Phi(x) ) = F( p(x), \Phi(x) ) = x$ and using similar arguments to calculate $\eta''(\tau)$ we find that
\begin{equation*}
  \eta_k'(0) = - \Phi(x) \frac{ \Phi_{x_k}(x)}{ \abs{ \nabla \Phi(x) }^2},
  \quad
  \eta_k''(0) = \Phi(x)^2 \, \sum_{l,r=1}^{n+1} \left(
    \delta_{kr} - \frac{2 \Phi_{x_k}(x) \Phi_{x_r}(x)}{ \abs{ \nabla \Phi(x) }^2}
    \right)
    \frac{ \Phi_{x_l}(x) \Phi_{x_l x_r}(x)}{\abs{ \nabla \Phi(x) }^4},
\end{equation*}
$k=1,\ldots,n+1$. Since $\eta(1)=p(x), \eta(0)=x$ we deduce with the help of Taylor's theorem that
\begin{equation}
  \label{pk}
  \begin{aligned}
    p_k(x) & = x_k - \Phi(x) \frac{\Phi_{x_k}(x)}{| \nabla \Phi(x) |^2} + \frac{1}{2} \Phi(x)^2 \sum_{l,r=1}^{n+1} \Bigl( \delta_{kr}-
    \frac{2 \Phi_{x_k}(x) \Phi_{x_r}(x)}{| \nabla \Phi(x) |^2} \Bigr) \frac{\Phi_{x_l}(x) \Phi_{x_l x_r}(x)}{| \nabla \Phi(x) |^4} \\
    & \quad + \Phi(x)^3 r_k(x), \quad k=1,\ldots,n+1,
  \end{aligned}
\end{equation}
where $r_k$ are smooth functions. In a similar way we may write
\begin{equation}
  \label{nablaphi}
  \nabla \Phi(x) = \nabla \Phi(p(x)) + \Phi(x) G(x),
\end{equation}
where $G(x) = \int_0^1 D^2 \Phi(p(x),\tau \Phi(x)) \frac{\partial F}{
\partial s}(p(x),\tau \Phi(x)) \dd \tau$.

Let us next use the function $p$ in order to define a particular extension of $z:\Gamma \rightarrow \R$ by
\begin{equation}
  \label{extension}
  z^e(x):= z(p(x)), \quad x \in U_{\delta}.
\end{equation}
Since $p( F(p(x), s ) ) = p(x)$ we deduce that $s \mapsto z^e(F(p(x),s))$ is independent of $s$ and therefore
\begin{equation}
  \label{zetang}
  \nabla z^e(x) \cdot \nu(x) = 0, \quad x \in U_{\delta}.
\end{equation}
In order to express the derivatives of $z^e$ in terms of the tangential derivatives of $z$ we first deduce from \eqref{pk} that
\begin{align*}
  p_{k,x_i}(x) = & \delta_{ik} - \frac{\Phi_{x_k}(x) \Phi_{x_i}(x)}{\abs{ \nabla \Phi(x) }^2} - \frac{ \Phi(x) \Phi_{x_k x_i}(x)}{\abs{ \nabla \Phi(x) }^2} +
  2 \Phi(x) \Phi_{x_k}(x) \sum_{l=1}^{n+1} \frac{ \Phi_{x_l}(x) \Phi_{x_l x_i}(x)}{ \abs{ \nabla \Phi(x) }^4} \\
  & + \Phi(x) \Phi_{x_i}(x) \sum_{l,r=1}^{n+1} \bigl( \delta_{kr} - \frac{2 \Phi_{x_k}(x) \Phi_{x_r}(x)}{ \abs{ \nabla \Phi(x) }^2} \bigr)
  \frac{\Phi_{x_l}(x) \Phi_{x_l x_r}(x)}{ \abs{ \nabla \Phi(x) }^4} + \Phi(x)^2 \alpha^i_k(x).
\end{align*}
Combining this relation with \eqref{nablaphi} we deduce that
\begin{align}
  \label{pk1}
  p_{k,x_i}(x) & = \delta_{ik} - \nu_i(p(x)) \nu_k(p(x)) + a_{ik}(x) \Phi(x), \\[2mm]
  p_{k,x_i x_j}(x) & = - \frac{\Phi_{x_i}(x) \Phi_{x_k x_j}(x)}{| \nabla \Phi(x) |^2} -
  \frac{\Phi_{x_j}(x) \Phi_{x_k x_i}(x)}{| \nabla \Phi(x) |^2} \nonumber  \\
  \label{pk2}
  & \quad + \frac{\Phi_{x_i}(x) \Phi_{x_j}(x)}{| \nabla \Phi(x) |^2} \sum_{l=1}^{n+1} \frac{\Phi_{x_l}(x) \Phi_{x_k x_l}(x)}{| \nabla \Phi(x) |^2}  + \beta^{ij}_k(x) \nu_k(p(x)) + \gamma^{ij}_k(x) \Phi(x)
\end{align}
where $a_{ik}, \beta^{ij}_k, \gamma^{ij}_k$ are smooth functions. Differentiating \eqref{extension} and using \eqref{pk1}, \eqref{pk2} as well as the fact that $\sum_{k=1}^{n+1} \underline{D}_k z(p(x)) \nu_k(p(x))=0$ we obtain
\begin{align}
  \label{gradze}
  & \nabla z^e(x) = \bigl( I + \Phi(x) A(x) \bigr) \nabla_{\Gamma} z(p(x)), \\
  \label{d2ze}
  & \frac{1}{| \nabla \Phi(x)|} \nabla \cdot \bigl( | \nabla \Phi(x)| \nabla z^e(x) \bigr) \\
  & \qquad = (\Delta_{\Gamma} z)(p(x)) + \Phi(x) \Bigl( \sum_{k,l=1}^{n+1} b_{lk}(x) \underline{D}_l \underline{D}_k z(p(x)) + \sum_{k=1}^{n+1} c_k(x) \underline{D}_k z(p(x)) \Bigr),
  \nonumber
\end{align}
where $A=(a_{ik}),b_{lk}$ and $c_k$ are again smooth.

\subsection{Bulk finite element space and inequalities}

In what follows we assume that the set $U$ is polyhedral. Let $(\T_h)_{0<h \leq h_0} $ be a  family of triangulations consisting of closed simplices $T$ with maximum mesh size $\displaystyle h:= \max_{T \in \mathcal{T}_h}h(T)$, where $h(T)= \operatorname{diam}(T)$. We assume that $(\T_h)_{0<h \leq h_0} $ is regular in the sense that there exists $\rho>0$ such that
\begin{equation}
  \label{regular}
  \operatorname{diam} B_T \geq \rho h(T), \quad \mbox{ for all } T \in \T_h, \, 0 < h \leq h_0,
\end{equation}
where $B_T$ is the largest ball contained in $T$. Let us denote by $X_h$ the space of linear finite elements
\begin{equation*}
  X_h = \lbrace \phi_h \in C^0(\bar{U}) \, | \, \phi_{h |T} \in P_1(T), T \in \mathcal{T}_h \rbrace,
\end{equation*}
and by $I_h: C^0(\bar{U})  \rightarrow X_h$ the usual Lagrange interpolation operator. We have,
for $\eta\in W^{2,p}(U)$
\begin{equation}
  \label{interpol}
  \norm{ \eta - I_h \eta }_{W^{k,p}(T)} \leq C h(T)^{2-k} \norm{ \eta}_{W^{2,p}(T)}, \quad T \in \mathcal{T}_h,
\end{equation}
for $ k=0,1$ and $1 < p \leq  \infty$ with $2-\frac{n+1}{p}>0$.  As a consequence,
\begin{equation}
  \label{interpolphi}
  \norm{ \Phi - I_h \Phi }_{L^{\infty}(U)} + h \norm{ \nabla( \Phi - I_h \Phi) }_{L^{\infty}(U)} \leq C h^2,
\end{equation}
so that we may assume that there exist constants $c_0,c_1$ such that
\begin{equation}
  \label{nablaihd}
  c_0 \leq | \nabla I_h \Phi(x) | \leq c_1, \quad x \in U, 0<h \leq h_0.
\end{equation}
Let us next define
\begin{align*}
  \Gamma_h &  := \lbrace x \in U \, | \, I_h \Phi(x) = 0 \rbrace  \\
  D_h & := \lbrace x \in U \, | \, \abs{ I_h \Phi (x) } < h \rbrace,
\end{align*}
as approximations of the given hypersurface $\Gamma$ and the neighbourhood $D^h:=\lbrace x \in U \, | \, \abs{ \Phi(x) } < h \rbrace$ ; see Figure~\ref{fig:domains} for example. Note that $\Gamma_h$ is a polygon whose facets are line segments if $n=1$ and a polyhedral surface whose facets consist of triangles or quadrilaterals if $n=2$.  The corresponding decomposition of $\Gamma_h$ is in general  not shape regular and can have arbitrary small elements.

\begin{figure}[hbtp]
  \begin{center}
    \includegraphics[width=0.9\textwidth]{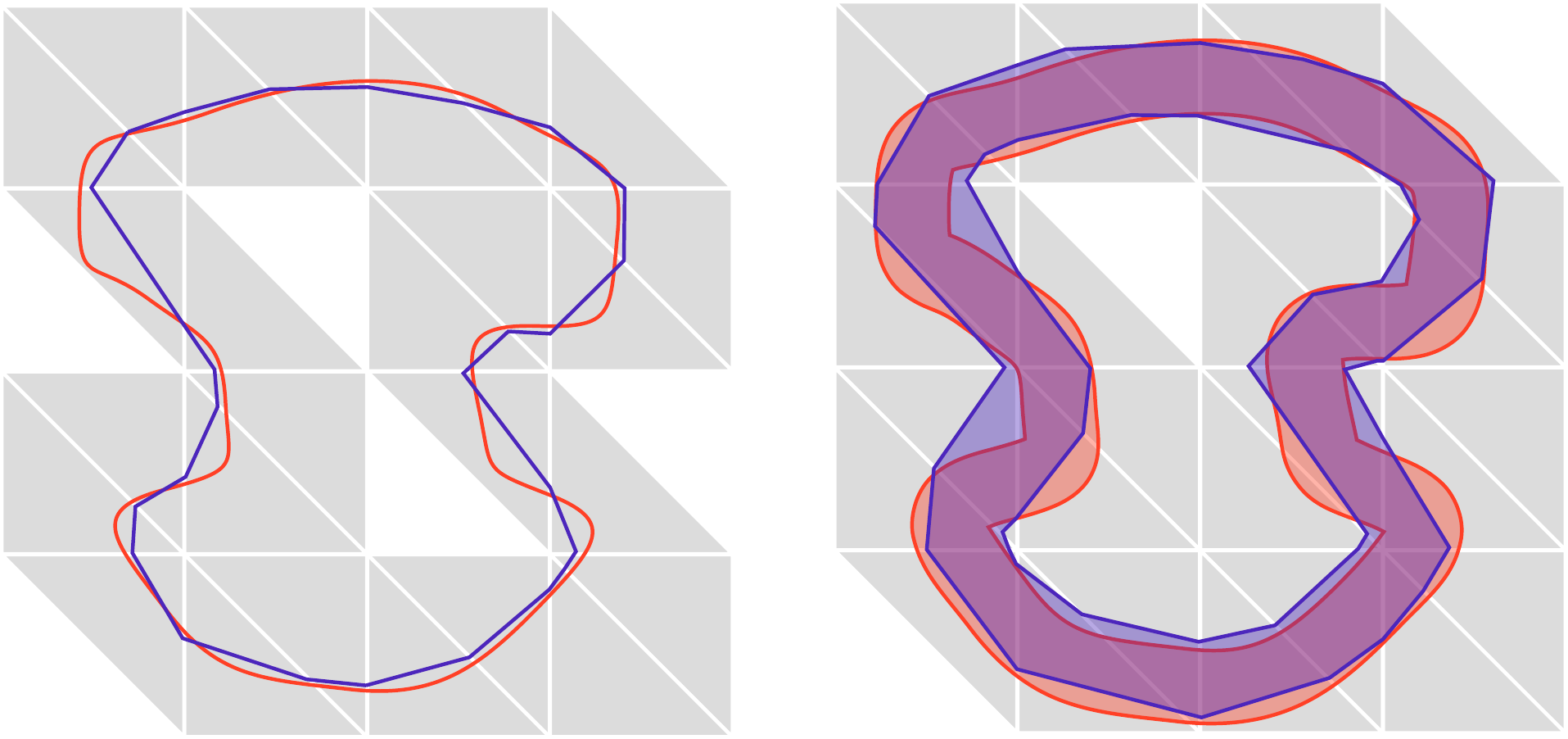}
  \end{center}

  \caption{A cartoon of the domains of the sharp interface (left) and the narrow band (right) method.
    The surface $\Gamma$ resp. the set $D^h$ is displayed in red, the approximations
    $\Gamma_h$ resp. $D_h$ in blue and the domains $U^I_h, U^B_h$ in grey.}
  \label{fig:domains}
\end{figure}

Furthermore, we introduce $F_h:U \rightarrow \R^{n+1}$ by
\begin{equation*}
  F_h(x):= F(p(x), I_h \Phi(x)), \quad x \in U,
\end{equation*}
where $F$ was defined in \eqref{defF}. From the properties of $F$ we infer that
\begin{eqnarray}
  p(F_h(x)) = p(x) \quad \mbox{ and } \quad \Phi(F_h(x)) = I_h \Phi(x), && \mbox{ if } F_h(x) \in U_{\delta}, \label{pdfh} \\
  F_h(x) = p(x), && \mbox{ if } x \in \Gamma_h. \label{fhp}
\end{eqnarray}

The following lemma collects the relevant properties of $F_h$.
\begin{lemma}
  \label{fhproperties}
  There exists $0<h_1 \leq h_0$ such that for $0 < h \leq h_1$ the mapping $F_h:D_h \rightarrow D^h: =  \lbrace x \in U \, | \, \abs{ \Phi(x) } < h \rbrace$ is bilipschitz with $F_h(\Gamma_h)=\Gamma$. Furthermore,
  \begin{align}
    \label{fhest1}
    \norm{ F_h - \id }_{L^{\infty}(U)} + h \norm{ DF_h - I }_{L^{\infty}(U)} & \leq c h^2,  \\
    \label{fhest2}
    \norm{ \abs{ \operatorname{det} DF_h } - \frac{ \abs{ \nabla I_h \Phi } }{ \abs{ \nabla \Phi } }  }_{L^{\infty}(U)}  & \leq  c h^2.
  \end{align}
\end{lemma}
\begin{proof}
  Since $F(p(x),\Phi(x))=x$ we deduce with the help of \eqref{interpolphi}
  \begin{equation*}
    | F_h(x) - x |  =  | F(p(x),I_h \Phi(x)) - F(p(x), \Phi(x)) |  \leq  c | I_h \Phi(x) - \Phi(x) | \leq c h^2.
  \end{equation*}
  Differentiating the relation $F_i(p(x),\Phi(x))=x_i$, $i=1,\ldots,n+1$ we obtain
  \begin{equation*}
    \sum_{k=1}^{n+1} \underline{D}_k F_i(p(x),\Phi(x)) p_{k,x_j}(x) + \frac{\partial F_i}{\partial s}(p(x),\Phi(x)) \Phi_{x_j}(x) = \delta_{ij}, \quad i,j=1,\ldots,n+1,
  \end{equation*}
  and hence
\begin{equation}
  \label{fhixj}
  \begin{aligned}
    F_{hi,x_j}(x) & = \sum_{k=1}^{n+1} \underline{D}_k F_i(p(x),I_h \Phi(x)) p_{k,x_j}(x) + \frac{\partial F_i}{\partial s}(p(x),I_h \Phi(x)) (I_h \Phi)_{x_j}(x) \\
    & = \delta_{ij} + \frac{\partial F_i}{\partial s}(p(x), \Phi(x)) \bigl(I_h \Phi - \Phi \bigr)_{x_j}(x) \\
    & \quad + \sum_{k=1}^{n+1} \Bigl( \underline{D}_k F_i(p(x),I_h \Phi(x)) -  \underline{D}_k F_i(p(x),\Phi(x)) \Bigr)  p_{k,x_j}(x) \\
    & \quad + \Bigl( \frac{\partial F_i}{\partial s}(p(x),I_h \Phi(x)) -  \frac{\partial F_i}{\partial s}(p(x),\Phi(x)) \Bigr) (I_h \Phi)_{x_j}(x) \\
    & = \delta_{ij} + \frac{\Phi_{x_i}(x)}{| \nabla \Phi(x) |^2} \bigl(I_h \Phi - \Phi \bigr)_{x_j}(x) + r_{ij}(x),
  \end{aligned}
\end{equation}
where $|r_{ij}(x)| \leq ch^2$ in view of \eqref{interpolphi}. This implies \eqref{fhest1}. In particular we deduce that $F_h$ is bilipschitz provided that $h$ is sufficiently small, whereas the properties $F_h(D_h)=D^h$ and $F_h(\Gamma_h)=\Gamma$ follow from \eqref{pdfh}. Finally we deduce from (\ref{fhixj}) that
  \begin{align*}
    \abs{ \operatorname{det} DF_h } & = 1 +  \frac{\nabla \Phi}{| \nabla \Phi |^2}  \cdot \nabla (I_h \Phi - \Phi) + c_h =
    \frac{ \nabla \Phi \cdot \nabla I_h \Phi}{| \nabla \Phi |^2}  + c_h \\
    & = \frac{| \nabla I_h \Phi| }{ | \nabla \Phi |} - \frac{1}{2} \abs{  \frac{ \nabla I_h \Phi }{
        | \nabla I_h \Phi| } -  \frac{ \nabla \Phi}{| \nabla \Phi |} }^2 \frac{| \nabla I_h \Phi|}{| \nabla \Phi |}
    + c_h = \frac{| \nabla I_h \Phi| }{ | \nabla \Phi |}  + d_h,
  \end{align*}
  where $\abs{c_h}, \abs{d_h} \leq ch^2$ proving \eqref{fhest2}.
  \qquad
\end{proof}

Next, let us introduce $\mu_h: \Gamma_h \rightarrow \R$ via $\mathrm{d} \sigma(p(x)) = \mu_h(x) \dd \sigma_h(x)$. It is well--known that
\begin{equation}
  \label{measure}
  \abs{ 1 - \mu_h } \le c h^2 \quad \mbox{ on } \Gamma_h.
\end{equation}

Using the properties of $F_h$ together with the coarea formula and \eqref{pk1},\eqref{pk2}, \eqref{gradze}, \eqref{measure} one can prove the following result on the equivalence of certain norms.
\begin{lemma}
  \label{lem:norm-equiv}
  There exist constants $c_1, c_2 > 0$ which are independent of $h$, such that for all $z \in H^1(\Gamma)$
  \begin{align*}
    c_1 \norm{ z^e  }_{L^2(\Gamma_h)} & \le \norm{ z }_{L^2(\Gamma)} \le  c_2 \norm{ z^e  }_{L^2(\Gamma_h)} \\
    c_1 \frac{1}{\sqrt{h}} \norm{ z^e }_{L^2(D_h)} & \le  \norm{ z }_{L^2(\Gamma)} \le c_2 \frac{1}{\sqrt{h}} \norm{ z^e }_{L^2(D_h)} \\
    c_1 \norm{ \nabla z^e  }_{L^2(\Gamma_h)} & \le  \norm{ \nabla_\Gamma z }_{L^2(\Gamma)} \le c_2 \norm{ \nabla z^e }_{L^2(\Gamma_h)} \\
    c_1 \frac{1}{\sqrt{h}} \norm{ \nabla z^e }_{L^2(D_h)} & \le  \norm{ \nabla_\Gamma z }_{L^2(\Gamma)} \le c_2 \frac{1}{\sqrt{h}}  \norm{ \nabla z^e }_{L^2(D_h)}.
  \end{align*}
  If in addition $z \in H^2(\Gamma)$ then
  \begin{equation*}
    c_1 \frac{1}{\sqrt{h}} \norm{ D^2 z^e }_{L^2(D_h)} \le  \norm{ z }_{H^2(\Gamma)}.
  \end{equation*}
\end{lemma}

\subsection{Variational form of elliptic equation and Strang's second lemma}

It is well--known \cite{Aub82} that for every $f \in L^2(\Gamma)$ there exists a unique solution $u \in H^2(\Gamma)$ of \eqref{eq:poisson} which satisfies
\begin{equation}
  \label{eq:regularity}
  \norm{ u }_{H^2(\Gamma)} \le c \norm{ f }_{L^2(\Gamma)}.
\end{equation}

Let us write \eqref{eq:poisson} in weak form
\begin{equation}
  \label{eq:weak-form}
  a(u,\varphi) = l(\varphi) \qquad \mbox{ for all } \varphi \in H^1(\Gamma),
\end{equation}
where
\begin{equation*}
  a( w, \varphi ) = \int_\Gamma \bigl( \nabla_\Gamma w \cdot \nabla_\Gamma \varphi + w \varphi \bigr) \dd \sigma,
  \quad
  l( \varphi ) = \int_\Gamma f \varphi \dd \sigma.
\end{equation*}

Next, suppose that $V_h$ is a finite--dimensional space and $V^e:= \lbrace v^e \, | \, v \in H^1(\Gamma) \rbrace$. Assume that $a_h: (V_h +  V^e) \times (V_h + V^e) \rightarrow \R$ is a symmetric, positive semidefinite bilinear form  which is in addition positive definite on $V_h \times V_h$.  Furthermore, let $l_h: V_h \rightarrow \R$ be linear. Then the approximate problem
\begin{equation}
  \label{approxeq}
  a_h(u_h,v_h) = l_h(v_h) \qquad \mbox{ for all } v_h \in V_h
\end{equation}
has a unique solution $u_h \in V_h$. Introducing
\begin{equation*}
  \norm{ v }_h := \sqrt{a_h(v,v)}, \qquad v \in V_h+V^e
\end{equation*}
we have by Strang's Second Lemma
\begin{equation}
  \label{strang}
  \norm{ u^e - u_h }_h \leq 2 \inf_{v_h \in V_h} \norm{ u^e - v_h }_h + \sup_{\phi_h \in V_h} \frac{ | a_h(u^e,\phi_h) - l_h(\phi_h) |}{\norm{ \phi_h }_h}.
\end{equation}

In the following two sections we shall present two differenct choices of $a_h$ and $l_h$ along with the corresponding analysis of the resulting schemes.


\section{Sharp interface method (\emph{SIF})}
\label{SIF}

\subsection{Setting up the method}

Let us begin by observing that if $T \in \T_h$ satisfies $\Hm^n( T \cap \Gamma_h ) > 0$, then the following two cases can occur:
\begin{enumerate}
\item $\Gamma_h \cap \mbox{int}(T) \neq \emptyset$, in which case $\Hm^n (\partial T \cap \Gamma_h)=0$;
\item $T \cap \Gamma_h = \partial T \cap \Gamma_h$ in which case $T \cap
\Gamma_h$ is the face between two elements.
\end{enumerate}
We may now define a unique subset $\T^I_h \subset \T_h$ by taking all elements satisfying case 1 and in case 2 taking just one of the two  elements $T$. The numerical method does not depend on which element is chosen. We may therefore conclude that there exists $N \subset \Gamma_h$ with $\Hm^n (N)=0$ and a subset $\T^I_h \subset \T_h$ such that every $x \in \Gamma_h \setminus N$ belongs to exactly one $T \in \T^I_h$. We then define
\begin{equation*}
  U_h^I = \bigcup_{T \in \T_h^I} T.
\end{equation*}
Clearly  $U_h^I \subseteq U_{\delta}$ provided that $h$ is small enough.  We define the finite element space $V_h^I$ by
\begin{equation*}
  V_h^I = \{ \phi_h \in C^0(U_h^I)  \, | \, \phi_{h|T} \in P_1(T) \mbox{ for each }  T \in \T_h^I \}.
\end{equation*}
Note that $\nabla \phi_h$ is defined on $\Gamma_h \setminus N$ in view of the definition of $\T_h^I$. In particular the unit normal $\nu_h$ to $\Gamma_h$ is given by
\begin{equation}
  \label{nuh}
  \nu_h = \frac{ \nabla I_h \Phi}{| \nabla I_h \Phi |} \quad \mbox{ on } \Gamma_h \setminus N,
\end{equation}
and we use \eqref{nuh} in order to extend $\nu_h$ to $U^I_h$. Let us next turn to the approximation error for the space $V_h^I$. Note that for a function  $z \in H^2(\Gamma)$ we have $z^e \in C^0(\bar{U}_{\delta})$ so that $I_h z^e$ is well--defined.

\begin{lemma}
  \label{interpol1}
  Let $z \in H^2(\Gamma)$. Then
  \begin{equation}
    \norm{ z^e - I_h z^e }_{L^2(\Gamma_h)}   + h \norm{ \nabla ( z^e - I_h z^e ) }_{L^2(\Gamma_h)}
    \le   c h^2 \norm{ z }_{H^2(\Gamma)}.
  \end{equation}
\end{lemma}

\begin{proof}
 We first observe that Theorem~3.7 in \cite{OlsReuGra09} yields
  \begin{equation}
    \norm{ z^e - I_h z^e }_{L^2(\Gamma_h)}
    + h \norm{ \nabla_{\Gamma_h} ( z^e - I_h z^e ) }_{L^2(\Gamma_h)}
    \le c h^2 \norm{ z }_{H^2(\Gamma)}.
  \end{equation}
  Hence, it remains to bound $ \norm{ \nabla (z^e - I_h z^e ) \cdot \nu_h }_{L^2(\Gamma_h)}$. To do so, we start by considering an element $T \in \T_h^I$. Then we see that
  \begin{align*}
    & \int_{T \cap \Gamma_h} \abs{ \nabla (z^e - I_h z^e) \cdot \nu_h }^2  \dd \sigma_h \\
    & \quad \leq 2 \int_{T \cap \Gamma_h} \abs{ \nabla z^e \cdot \nu_h }^2 \dd \sigma_h
    + 2 \int_{T \cap \Gamma_h} \abs{ \nabla (I_h z^e) \cdot \nu_h }^2 \dd \sigma_h \\
    & \quad \leq 2  \int_{T \cap \Gamma_h} \abs{ \nabla z^e \cdot ( \nu_h - \nu) }^2 \dd \sigma_h +
    c h(T)^{-1} \int_T \abs{ \nabla (I_h z^e) \cdot \nu_h }^2 \dd x
    = : I_1 + I_2,
  \end{align*}
  in view of \eqref{zetang} and the fact that $\Hm^{n}(T \cap \Gamma_h) \leq c h(T)^{-1} \Hm^{n+1}(T)$. Note that by \eqref{nuh} and \eqref{interpolphi}
  \begin{equation}
    \label{numinnuh}
    \norm{ \nu - \nu_h }_{L^{\infty}(T)} = \norm{ \frac{\nabla \Phi}{ \abs{ \nabla \Phi } } - \frac{\nabla I_h \Phi}{ \abs{ \nabla  I_h \Phi } } }_{L^{\infty}(T)} \leq c h(T)
  \end{equation}
  so that
  \begin{equation*}
    I_1 \le c h^2 \int_{T \cap \Gamma_h} \abs{ \nabla z^e }^2 \dd \sigma_h.
  \end{equation*}
  Furthermore, recalling \eqref{interpol} and using again \eqref{numinnuh}
  \begin{equation*}
    I_2   \le c h(T)^{-1} \int_T \bigl( \abs{ \nabla z^e \cdot (\nu_h - \nu) }^2 +
    \abs{ \nabla (z^e - I_h z^e ) }^2 \bigr)  \dd x   \le  c h \norm{ z^e }_{H^2(T)}^2.
  \end{equation*}
  We use the bounds for $I_1, I_2$ and sum over all elements $T \in \T_h^I$, then apply Lemma~\ref{lem:norm-equiv} to see
  \begin{equation*}
    \int_{\Gamma_h} \abs{ \nabla (z^e - I_h z^e) \cdot \nu_h }^2 \dd \sigma_h
    \le c h^2 \Vert \nabla z^e \Vert_{L^2(\Gamma_h)}^2 + c h \norm{ z^e }_{H^2(D_{c_1 h})}^2
    \le c h^2 \norm{ z }_{H^2(\Gamma)}^2,
  \end{equation*}
  since $T \subset D_{c_1 h}$ for all $T \in \mathcal T_h^I$ in view of \eqref{nablaihd}.
  \qquad
\end{proof}

\subsection{The method}

Let us write \eqref{sim} in the form:- Find $u_h \in V_h^I$ such that
\begin{equation}
  \label{sim1}
  a_h(u_h,\phi_h) = l_h(\phi_h)
  \quad \mbox{ for all } \phi_h \in V_h^I,
\end{equation}
where
\begin{equation*}
  a_h ( w_h, \phi_h )
  = \int_{\Gamma_h} \bigl( \nabla w_h \cdot \nabla \phi_h + w_h \phi_h \bigr) \dd \sigma_h,
  \quad
  l_h ( \phi_h )
  = \int_{\Gamma_h} f^e \phi_h \dd \sigma_h.
\end{equation*}

In order to verify that the symmetric bilinear form $a_h$ is positive definite on $V_h^I \times V_h^I$ we note that $a_h(\phi_h,\phi_h) = 0$ implies that
\begin{equation*}
  \int_{\Gamma_h \cap T} \bigl( \abs{ \nabla \phi_h }^2 + \phi_h^2 \bigr) \dd \sigma_h = 0 \quad \mbox{ for all } T \in \T_h^I.
\end{equation*}
Since $\Hm^n( T \cap \Gamma_h ) > 0$ for $T \in \T_h^I$ we infer that $\nabla \phi_h=0$ and hence $\phi_h$ is constant on these elements. Using again that $\Hm^n( T \cap \Gamma_h ) > 0$  we deduce that $\phi_h=0$ on each $T \in \T_h^I$ so that $\phi_h \equiv 0$ in $V_h^I$. Hence \eqref{sim1} has a unique solution $u_h \in V_h^I$ which satisfies
\begin{equation}
  \norm{ u_h }_h = \bigl( \norm{ \nabla u_h }_{L^2(\Gamma_h)}^2 + \norm{ u_h }_{L^2(\Gamma_h)}^2 \bigr)^{\frac{1}{2}}
  \leq c \norm{ f^e }_{L^2(\Gamma_h)}   \le c \norm{ f }_{L^2(\Gamma)}.
\end{equation}
\begin{remark}
The right hand side $l_{h}(\cdot)$ may be defined using other sufficiently accurate extensions of $f$.
\end{remark}
\subsection{Error analysis}

\begin{theorem}
  \label{thm:line-bound}
  Let $u$ be the solution of \eqref{eq:poisson} and $u_h$ the solution of the finite element scheme \eqref{sim1}. Then
  \begin{equation}
    \norm{ u^e - u_h }_{L^2(\Gamma_h)} + h \norm{ \nabla (u^e - u_h) }_{L^2(\Gamma_h)}
    \le c h^2 \norm{ f }_{L^2(\Gamma)}.
  \end{equation}
\end{theorem}

\begin{proof}
  In view of the definition of $\norm{ \cdot }_h$, \eqref{strang} and  Lemma~\ref{interpol1} we have for $e_h:=u^e-u_h$
\begin{eqnarray}
\lefteqn{\hspace{1cm}  \bigl(  \norm{ e_h }^2_{L^2(\Gamma_h)} + \norm{ \nabla e_h }^2_{L^2(\Gamma_h)} \bigr)^{\frac{1}{2}} } \label{ee} \\
& \leq &  2 \bigl(  \norm{ u^e - I_h u^e }^2_{L^2(\Gamma_h)} + \norm{ \nabla (u^e - I_h u^e) }^2_{L^2(\Gamma_h)} \bigr)^{\frac{1}{2}}
    + \sup_{\phi_h \in V_h^I} \frac{ | a_h(u^e,\phi_h) - l_h(\phi_h) |}{\Vert \phi_h \Vert_h} \nonumber  \\
& \leq &  c h \Vert u \Vert_{H^2(\Gamma)} + \sup_{\phi_h \in V_h^I} \frac{ | a_h(u^e,\phi_h) - l_h(\phi_h) |}{\Vert \phi_h \Vert_h}. \nonumber
\end{eqnarray}
  In order to estimate the second term we let $\phi_h \in V_h^I$ be arbitrary and define $\varphi_h:= \phi_h \circ F_h^{-1}$.  Then
  \begin{equation*}
    a_h(u^e,\phi_h) - l_h(\phi_h) = \bigl( a_h(u^e,\phi_h)-a(u,\varphi_h) \bigr) + \bigl( l(\varphi_h) -
    l_h(\phi_h) \bigr)  \equiv  I + II.
  \end{equation*}
 
  Using the transformation rule and \eqref{gradze} we obtain
\begin{eqnarray}
\lefteqn{ \hspace{-1.5cm}   \int_\Gamma \bigl( \nabla_\Gamma u \cdot \nabla_\Gamma \varphi_h + u \varphi_h \bigr) \dd \sigma
  = \int_{\Gamma_h} \bigl( (\nabla_{\Gamma} u) \circ p \cdot  (\nabla_{\Gamma} \varphi_h) \circ p + (u \circ p) \, (\varphi_h \circ p) \bigl) \mu_h \dd \sigma_h } \nonumber \\
& = &  \int_{\Gamma_h} \bigl(   (I + \Phi A)^{-1} \nabla u^e \cdot (I + \Phi A)^{-1} \nabla \varphi_h^e + u^e \, \varphi_h^e  \bigr) \mu_h 
\dd \sigma_h. \label{ahma}
\end{eqnarray}
Since $\varphi_h^e( x ) = \varphi_h(p(x))= \phi_h( F_h^{-1}( p( x ) ) )$ we derive
  \begin{align*}
    \nabla \varphi_h^e ( x )
    & = [ D p( x ) ]^T [ DF_h^{-1} ( p ( x ) ) ]^T \nabla \phi_h ( F_h^{-1} ( p ( x ) ) ) \\
    & = [ D p( x ) ]^T [ DF_h ( F_h^{-1} ( p ( x ) ) ) ]^{-T}\nabla \phi_h ( F_h^{-1}( p ( x ) ) ).
  \end{align*}
  We infer from (\ref{fhixj}) that
  \begin{equation}
    \label{eq:DFh-T} 
    (DF_h)^{-T} = I - \frac{1}{| \nabla \Phi |} \nabla \eta_h \otimes \nu + B_h, \quad \mbox{ with } \abs{ B_h } \leq c h^2,
  \end{equation}
where $\eta_h = I_h \Phi - \Phi$. It follows from (\ref{fhp}) that $F_h^{-1}(p(x))=x, x \in \Gamma_h$, which together with (\ref{pk1})
implies
\begin{equation*}
\nabla \varphi_h^e  = ( I - \nu \otimes \nu ) ( I - \frac{1}{\abs{ \nabla \Phi }} \nabla \eta_h \otimes \nu ) \nabla \phi_h + q_h  \quad \mbox{ on } \Gamma_h, \quad 
| q_h | \leq c h^2 | \nabla \phi_h |.
\end{equation*}
Taking into account that $\nabla u^e \cdot \nu = 0$ we therefore have 
  \begin{equation*}
  \nabla u^e \cdot \nabla \varphi_h^e  = \nabla u^e \cdot \nabla \phi_h - \frac{1}{\abs{\nabla \Phi}} ( \nabla u^e \cdot \nabla \eta_h ) ( \nabla \phi_h \cdot \nu ) + \nabla u^e \cdot q_h \quad \mbox{ on } \Gamma_h.
 \end{equation*}
If we insert this relation into (\ref{ahma}) and recall the definition of $a_h$ we find that
\begin{eqnarray*}
\abs{ I } & \le &  \int_{\Gamma_h} \bigl( \abs{ \nabla u^e \cdot \nabla \phi_h - \mu_h (I + \Phi A)^{-T} ( I + \Phi A)^{-1}  \nabla u^e \cdot 
\nabla \varphi_h^e} +  \abs{ ( \mu_h -1 )u^e \varphi_h^e } \bigr) \dd \sigma_h \\
& \le &  c h^2 \norm{ u^e }_h \norm{ \phi_h }_h + c h \Vert u^e \Vert_h  \int_{\Gamma_h} | ( \nabla \phi_h \cdot \nu ) |  \dd \sigma_h
\end{eqnarray*}
where we used \eqref{interpolphi}, \eqref{measure} and the fact that $\varphi_h^e=\phi_h$ on $\Gamma_h$. Similarly,
  \begin{align*}
    \abs{ II } & = \abs{ \int_\Gamma f \varphi_h \dd \sigma - \int_{\Gamma_h} f^e \phi_h \dd \sigma_h } \leq \int_{\Gamma_h} \abs{ 1 - \mu_h } |  f^e | \, |  \phi_h | \dd \sigma_h \\
    & \le c h^2 \norm{ f^e }_{L^2(\Gamma_h)} \norm{ \phi_h }_h \leq c h^2 \norm{ f }_{L^2(\Gamma)} \norm{ \phi_h }_h.
  \end{align*}

Combining these estimates with \eqref{eq:regularity} we have
\begin{equation} \label{supest}
| a_h(u^e,\phi_h) - l_h(\phi_h) | \leq c h^2 \Vert f \Vert_{L^2(\Gamma)}\norm{ \phi_h }_h 
+ c h \Vert f \Vert_{L^2(\Gamma)} \int_{\Gamma_h}  \abs{  \nabla \phi_h \cdot \nu  } \dd \sigma_h
\end{equation}
for all $\phi_h \in V^I_h$, which inserted into (\ref{ee}) yields
  \begin{equation}
    \label{ehest0}
    \norm{ e_h }_{L^2(\Gamma_h)} + \norm{ \nabla e_h }_{L^2(\Gamma_h)} \leq ch \norm{ f }_{L^2(\Gamma)}.
  \end{equation}
  In order to improve the $L^2$-error bound we employ the usual Aubin-Nitsche argument. Denote by $w \in H^2(\Gamma)$ the solution of the dual problem
  \begin{equation*}
    a( \varphi, w) = \int_{\Gamma} \tilde{e}_h \varphi \dd \sigma \qquad \mbox{ for all } \varphi \in H^1(\Gamma), \quad 
\mbox{ with } \tilde{e}_h= e_h \circ F_h^{-1},
  \end{equation*}
  which satisfies
  \begin{equation}
    \label{eq:dual-reg}
    \norm{ w }_{H^2(\Gamma)} \le c \norm{ \tilde{e}_h }_{L^2(\Gamma)}.
  \end{equation}
  We have in view of \eqref{sim}
  \begin{equation}
    \label{ehest1}
    \begin{aligned}
      \norm{  \tilde{e}_h }_{L^2(\Gamma)}^2 & = a(\tilde{e}_h,w) = \bigl( a(\tilde{e}_h,w) - a_h(e_h,w^e) \bigr) \\
      & \quad + a_h(e_h,w^e-I_h w^e) + \bigl( a_h(u^e,I_h w^e) - l_h(I_h w^e) \bigr) \\
      & \equiv I + II + III.
    \end{aligned}
  \end{equation}
  Similarly as above we deduce with the help of \eqref{ehest0} and Lemma~\ref{lem:norm-equiv}
  \begin{equation*}
    \abs{ I } \leq c h \norm{ e_h }_h \norm{ w^e }_h \leq c h^2 \norm{ f }_{L^2(\Gamma)} \norm{ w }_{H^1(\Gamma)}.
  \end{equation*}
  Next, Lemma~\ref{interpol1} and \eqref{ehest0} imply
  \begin{equation*}
    \abs{ II } \leq \norm{ e_h }_h \norm{ w^e - I_h w^e }_h \leq c h^2 \norm{ f }_{L^2(\Gamma)} \norm{ w }_{H^2(\Gamma)}.
  \end{equation*}
  Finally, \eqref{supest}, the fact that $\nabla w^e \cdot \nu=0$  and Lemma~\ref{interpol1} yield
  \begin{eqnarray*}
    \abs{ III } &  \leq &  c h^2 \norm{ f }_{L^2(\Gamma)} \norm{ I_h w^e }_h + c h \norm{ f }_{L^2(\Gamma)} 
\int_{\Gamma_h} \abs{  \nabla (I_h w^e - w^e) \cdot \nu  } \dd \sigma_h \\ 
 &  \leq &  c h^2 \norm{ f }_{L^2(\Gamma)} \norm{ w }_{H^2(\Gamma)}.
  \end{eqnarray*}
 Inserting the above estimates into \eqref{ehest1} and recalling \eqref{eq:dual-reg} we obtain
  \begin{equation*}
  \norm{ \tilde{e}_h}_{L^2(\Gamma)} \le c h^2 \norm{ f }_{L^2(\Gamma)},
  \end{equation*}
  which together with Lemma \ref{lem:norm-equiv} completes the proof  since $\tilde{e}_h^e=e_h$ on $\Gamma_h$.
  \end{proof}


\section{Narrow band method}
\label{NBM}

\subsection{Setting up the method}

Let us consider for $D_h = \lbrace x \in U_\delta \, | \, \abs{ I_h \Phi(x) } <h \rbrace$ the set
\begin{equation*}
  \T^B_h = \{ T \in \T_h \, | \,  \Hm^{n+1}( T \cap D_h ) > 0 \},
\end{equation*}
along with
\begin{equation*}
  U^B_h = \bigcup_{T \in \T^B_h} T.
\end{equation*}
We define the finite element space $V_h^B$ on the triangulation $\T^B_h$  by
\begin{equation*}
  V_h^B = \{ \phi_h \in C^0(U^B_h) \, | \,
  \phi_{h|T} \in P_1(T) \mbox{ for each } T \in \T^B_h \}.
\end{equation*}
Let us first examine the approximation error for the space $V_h^B$.
\begin{lemma}
  \label{interpol2}
  We have for each function $z \in H^2(\Gamma)$:
  \begin{equation}
    \frac{1}{\sqrt{h}} \norm{ z^e - I_h z^e }_{L^2(D_h)} +  \sqrt{h} \norm{ \nabla ( z^e - I_h z^e ) }_{L^2(D_h)}
    \le  c h^2 \norm{ z }_{H^2(\Gamma)}.
  \end{equation}
\end{lemma}

\begin{proof}
  We infer from \eqref{interpol} and Lemma~\ref{lem:norm-equiv} that
  \begin{align*}
    & \frac{1}{h} \Vert z^e - I_h z^e \Vert_{L^2(D_h)}^2 + h \Vert \nabla (z^e - I_h z^e) \Vert_{L^2(D_h)}^2 \\
    & \quad \leq \sum_{T \cap D_h \neq \emptyset} \bigl( \frac{1}{h} \Vert z^e - I_h z^e \Vert_{L^2(T)}^2 + h
    \Vert \nabla (z^e - I_h z^e) \Vert_{L^2(T)}^2 \bigr)  \\
    & \quad \leq c h^3 \sum_{T \cap D_h \neq \emptyset} \Vert z^e \Vert_{H^2(T)}^2 \leq c h^3 \Vert z^e \Vert_{H^2(D_{(1+c_1)h})}^2
    \leq c h^4 \Vert z \Vert_{H^2(\Gamma)}^2,
  \end{align*}
  since $T \subset D_{(1+c_1)h}$ for all $T \cap D_h \neq \emptyset$ in view of \eqref{nablaihd}.
  \qquad
\end{proof}

\subsection{The method}

Let us write \eqref{nbm} in the form:- Find $u_h \in V_h^B$ such that
\begin{equation}
  \label{nbm1}
  a_h(u_h,\phi_h) = l_h(\phi_h) \quad \mbox{ for all } \phi_h \in V_h^B,
\end{equation}
where
\begin{align*}
  a_h( w_h, \phi_h )
  & = \frac{1}{2 h} \int_{D_h} \bigl( \nabla w_h \cdot \nabla \phi_h + w_h \phi_h \bigr) \abs{ \nabla I_h \Phi } \dd x, \\
  l_h (\phi_h)
  & = \frac{1}{2 h} \int_{D_h} f^e \phi_h \abs{ \nabla I_h \Phi } \dd x.
\end{align*}

Note that the factors $\frac{1}{h}$ in each of the above terms are there  to aid the notation for
the error analysis. In a similar way as for the sharp interface method one can verify that $a_h$ is positive definite on $V_h^B \times V_h^B$.
Hence, the finite element scheme \eqref{nbm1} has a unique solution $u_h \in V_h^B$ which satisfies
\begin{equation}
  \label{eq:here}
  \norm{ u_h }_h =
  \left( \frac{1}{2 h} \int_{D_h} \bigl( \abs{ \nabla u_h }^2 + u_h^2 \bigr) | \nabla I_h \Phi |  \dd x \right)^{\frac12}
  \le c \norm{ f }_{L^2(\Gamma)}.
\end{equation}
\begin{remark}
The right hand side $l_{h}(\cdot)$ may be defined using other sufficiently accurate extensions of $f$.
\end{remark}

\subsection{Error analysis}

Before we prove our main error bound we formulate a technical lemma which will be helpful in the error analysis.

\begin{lemma}
  \label{aux}
  Suppose that $u \in H^2(\Gamma)$ is  a solution of (\ref{eq:poisson}). Then,
  \begin{equation}
    a_h(u^e,\phi) = \frac{1}{2h} \int_{D^h} f^e \phi \circ F_h^{-1} \abs{ \nabla \Phi} \dd x + \frac{1}{2h} \int_{D_h} ( \nabla u^e \cdot \nabla \eta_h ) ( \nabla \phi \cdot \nu ) \frac{ \abs{ \nabla I_h \Phi} }{ \abs{ \nabla \Phi } } \dd x + \langle S, \phi \rangle,
  \end{equation}
  for all $\phi \in H^1(D_h)$, where $\eta_h = I_h \Phi - \Phi$ and
  \begin{equation*}
    \abs{ \langle S, \phi \rangle } \leq C h^2 \norm{ u }_{H^2(\Gamma)} \norm{ \phi }_h.
  \end{equation*}
\end{lemma}

\begin{proof}
  To begin, we derive from \eqref{d2ze} and \eqref{eq:poisson} that
  \begin{equation}
    \label{extendedeq}
    - \frac{1}{| \nabla \Phi|} \nabla \cdot \bigl( | \nabla \Phi | \, \nabla u^e \bigr) + u^e = f^e + R \quad \mbox{ in } U_{\delta},
  \end{equation}
  where
  \begin{equation}
    \label{rest}
    R(x) = - \Phi(x) \Bigl( \sum_{k,l=1}^{n+1} b_{lk}(x) \underline{D}_l \underline{D}_k u(p(x)) + \sum_{k=1}^{n+1} c_k(x)
    \underline{D}_k u(p(x)) \Bigr).
  \end{equation}

  We multiply \eqref{extendedeq}  by $\phi \circ F_h^{-1} | \nabla \Phi |, \phi \in H^1(D_h)$ and integrate over $D^h$. Since $\dfrac{\partial u^e}{\partial \nu} = 0 $ on $\partial D^h$ we obtain after integration by parts
  \begin{equation}
    \label{transf}
    \begin{aligned}
      & \int_{D^h} \nabla u^e \cdot \nabla ( \phi \circ  F_h^{-1}) | \nabla \Phi | \, \dd x + \int_{D^h} u^e \, \phi \circ F_h^{-1}
      | \nabla \Phi | \, \dd x \\
      & \qquad = \int_{D^h} f^e \, \phi \circ F_h^{-1} | \nabla \Phi | \dd x + \int_{D^h} R \, \phi \circ F_h^{-1}
      | \nabla \Phi | \, \dd x.
    \end{aligned}
  \end{equation}
  Observing that $\nabla ( \phi \circ F_h^{-1}) = [(DF_h)^{-T} \circ F_h^{-1}] \nabla \phi \circ F_h^{-1}$ the transformation rule and Lemma~\ref{fhproperties} imply that
  \begin{equation*}
    I :=  \int_{D^h} \nabla u^e \cdot \nabla ( \phi \circ  F_h^{-1}) | \nabla \Phi | \, \dd x =
    \int_{D_h}  \nabla u^e \circ F_h \cdot (DF_h)^{-T} \nabla \phi \, | \nabla \Phi \circ F_h | \, \abs{ \operatorname{det} DF_h } \dd x.
  \end{equation*}
  Recalling \eqref{extension} and \eqref{pdfh} we have
  \begin{equation}
    \label{zetrans}
    z^e(x) = z(p(x)) = z(p(F_h(x)) = z^e(F_h(x)),
  \end{equation}
  from which we deduce by differentiation
  \begin{equation}
    \label{uefh}
    \nabla z^e \circ F_h = (DF_h)^{-T} \nabla z^e,
  \end{equation}
  so that
  \begin{equation*}
    I = \int_{D_h} ( D F_h )^{-T} \nabla u^e \cdot ( D F_h )^{-T} \nabla \phi \, \abs{ \nabla \Phi \circ F_h } \abs{ \operatorname{det} D F_h } \dd x.
  \end{equation*}
  Recalling \eqref{eq:DFh-T}, we find with the help of $\nabla u^e \cdot \nu=0$ that
  \begin{equation*}
    (DF_h)^{-T} \nabla u^e = \nabla u^e + B_h \nabla u^e, \quad (DF_h)^{-T} \nabla \phi = \nabla \phi - \frac{1}{| \nabla \Phi |} (\nabla \phi \cdot \nu)
    \nabla \eta_h + B_h \nabla \phi,
  \end{equation*}
  where $\abs{ B_h } \le c h^2$. Furthermore,  Lemma~\ref{fhproperties} implies that
  \begin{equation*}
    | \nabla \Phi \circ F_h | \, \abs{ \operatorname{det} DF_h } = | \nabla I_h \Phi | + \gamma_h, \quad \mbox{ where } | \gamma_h | \leq c h^2,
  \end{equation*}
  so that in conclusion
  \begin{equation*}
    I = \int_{D_h} \nabla u^e \cdot \nabla \phi \, \abs{ \nabla I_h \Phi } \dd x - \int_{D_h} ( \nabla u^e \cdot \nabla \eta_h)
    ( \nabla \phi \cdot \nu) \frac{| \nabla I_h \Phi|}{| \nabla \Phi|}  \dd x
    + \langle R^1_h,\phi \rangle,
  \end{equation*}
  where
  \begin{equation*}
    \abs{ \langle R^1_h,\phi \rangle } \leq c h^2 \norm{ \nabla u^e }_{L^2(D_h)} \norm{ \nabla \phi }_{L^2(D_h)} \leq c h^3  \norm{ u }_{H^2(\Gamma)} \norm{ \phi }_h,
  \end{equation*}
  in view of Lemma~\ref{lem:norm-equiv} and the definition of $\norm{ \cdot }_h$. Similarly, \eqref{zetrans} and \eqref{fhest2} yield
  \begin{equation*}
    \int_{D^h} u^e  \, \phi \circ F_h^{-1} | \nabla \Phi | \,   \dd x  =  \int_{D_h} u^e \phi \, \abs{ \nabla I_h \Phi } \dd x + \langle R^2_h,\phi \rangle
  \end{equation*}
  with $\abs{ \langle R^2_h,\phi \rangle } \leq c h^3 \norm{ u }_{H^2(\Gamma)} \norm{ \phi }_h$. Inserting the  above identities into \eqref{transf} and dividing by $2h$ we derive
  \begin{equation}
    \label{ahlh}
    \begin{aligned}
      a_h(u^e,\phi)  & = \frac{1}{2h} \int_{D^h} f^e \, \phi \circ F_h^{-1} | \nabla \Phi | \,\dd x + \frac{1}{2 h}
      \int_{D^h} R \, \phi \circ F_h^{-1} | \nabla \Phi | \,  \dd x \\
      & \qquad + \frac{1}{2h} \int_{D_h} ( \nabla u^e \cdot \nabla \eta_h) ( \nabla \phi \cdot \nu) \frac{| \nabla I_h \Phi |}{
        | \nabla \Phi |}  \dd x - \frac{1}{2h} \langle R^1_h,\phi \rangle - \frac{1}{2h} \langle R^2_h,\phi \rangle.
    \end{aligned}
  \end{equation}
  In order to rewrite the integral over $D^h$ we recall that $F$ is a diffeomorphism from $\Gamma \times (-h,h)$ onto $D^h$. It is not difficult to see that
  \begin{equation}
    \label{transf1}
    \mathrm{d} x = \mu(p,s) \dd \sigma_p \dd s, \;
    \mbox{ where } \abs{ \mu(p,s) - \frac{1}{| \nabla \Phi(F(p,s))| } }\leq C \abs{s}, \;
    \abs{s} < h, p \in \Gamma.
  \end{equation}
  Hence,
  \begin{align*}
    & \int_{D^h} R \, \phi \circ F_h^{-1} | \nabla \Phi | \, \dd x
    = \int_{-h}^h \int_{\Gamma} R \circ F \,  \phi \circ F_h^{-1} \circ F \; | \nabla \Phi \circ F | \, \mu  \dd \sigma_p \dd s \\
    & \quad = \int_{-h}^h \int_{\Gamma} R \circ F \, \phi \circ F_h^{-1} \circ F  \dd \sigma_p \dd s
    +  \int_{-h}^h \int_{\Gamma} \tilde{r} \,  \phi \circ F_h^{-1} \circ F  \dd \sigma_p \dd s \\
    & \quad \equiv T_1 + T_2,
  \end{align*}
  where  $\tilde{r}(p,s) = R(F(p,s)) \bigl( \mu(p,s) \, | \nabla \Phi(F(p,s))| - 1  \bigr)$.
  In order to treat $T_1$ we deduce from \eqref{rest} and the fact that $\Phi(F(p,s))=s$ that
  \begin{equation*}
    R(F(p,s))  = - s \Bigl( \sum_{k,l=1}^{n+1} b_{lk}(F(p,s)) \underline{D}_l \underline{D}_k u(p) + \sum_{k=1}^{n+1} c_k(F(p,s))
    \underline{D}_k u(p) \Bigr).
  \end{equation*}
  Since $\int_{-h}^h s \dd s =0$, the first term in $T_1$ can be written as
  \begin{equation*}
    -  \int_{-h}^h \int_{\Gamma} \sum_{k,l=1}^{n+1} s \underline{D}_l \underline{D}_k u(p) \Big\{ b_{lk} (F(p,s)) \, \phi \circ F_h^{-1}(F(p,s))
    - b_{lk}(p) \, \phi \circ F_h^{-1}(p) \Big\} \dd \sigma_p \dd s.
  \end{equation*}
  Treating the second term in $T_1$ in the same way and observing that $p=F(p,0)$ we deduce with the help of the fundamental theorem of calculus that
  \begin{equation*}
    \abs{ T_1 } \leq c h^{\frac{5}{2}} \norm{ u }_{H^2(\Gamma)} \left( \int_{D^h} \Bigl( \abs{ \nabla \phi \circ F_h^{-1} }^2 + \abs{ \phi \circ F_h^{-1} }^2 \Bigr) \dd x \right)^{\frac12}
    \leq ch^3 \norm{ u }_{H^2(\Gamma)} \norm{ \phi }_h.
  \end{equation*}
  Next, we infer from \eqref{rest} and \eqref{transf1} that
  \begin{equation*}
    \abs{ \tilde{r}(p,s) } \leq c s^2 \bigl( \abs{ \nabla_{\Gamma} u(p) } + \abs{ D^2_{\Gamma} u(p) } \bigr),
  \end{equation*}
  so that
  \begin{equation*}
    \abs{ T_2 } \leq C h^{\frac{5}{2}} \norm{ u }_{H^2(\Gamma)} \left( \int_{D^h} \abs{ \phi \circ F_h^{-1} }^2  \dd x \right)^{\frac12}
    \leq C h^3 \norm{ u }_{H^2(\Gamma)} \norm{ \phi }_h.
  \end{equation*}

  The result now follows from \eqref{ahlh} together with the bounds on $R^1_h$ and $R^2_h$.
  \qquad
\end{proof}

We are now in position to prove optimal error bounds for our scheme.

\begin{theorem}
  \label{thm:band-bound}
  Let $u$ be the solution of \eqref{eq:poisson} and $u_h$ the solution of the finite element scheme \eqref{nbm1}. Then
  \begin{equation}
    \norm{ u^e- u_h }_{L^2(\Gamma_h)} + h  \left( \frac{1}{2 h} \int_{D_h} \abs{ \nabla (u^e - u_h) }^2 \abs{ \nabla I_h \Phi } \dd x \right)^{\frac12}
    \leq c h^2 \norm{ f }_{L^2(\Gamma)}.
  \end{equation}
\end{theorem}

\begin{proof}
  Let us write $e_h:= u^e -u_h$. We infer from \eqref{strang} and Lemma~\ref{interpol2} that
  \begin{equation}
    \label{err1}
    \norm{ e_h }_h \leq  ch \norm{ u }_{H^2(\Gamma)} + \sup_{\phi_h \in V_h^B} \frac{ \abs{ a_h(u^e,\phi_h)
      - l_h(\phi_h) } }{\norm{ \phi_h }_h}.
  \end{equation}

  The second term on the right hand side can be estimated with the help of Lemma~\ref{aux}. The transformation rule together with (\ref{zetrans}) yields
  \begin{equation*}
    \frac{1}{2 h} \int_{D^h} f^e \phi_h \circ F_h^{-1} \abs{ \nabla \Phi } \dd x
    = \frac{1}{2h} \int_{D_h} f^e \circ F_h \, \phi_h \, \abs{ \nabla \Phi \circ F_h } \,  \abs{ \operatorname{det} DF_h } \dd x,
  \end{equation*}
  so that we deduce from Lemma~\ref{aux}
  \begin{align*}
    a_h(u^e,\phi_h) - l_h(\phi_h)
    = & \frac{1}{2h} \int_{D_h} ( \nabla u^e \cdot \nabla \eta_h) \, (\nabla \phi_h \cdot \nu) \frac{\abs{ \nabla I_h \Phi } }{ \abs{ \nabla \Phi } } \\
    &  + \frac{1}{2h} \int_{D_h} f^e \phi_h \Bigl( \abs{ \nabla \Phi \circ F_h } \, \abs{ \mbox{det} DF_h } - \abs{ \nabla I_h \Phi } \Bigr) \dd x + \langle S, \phi_h \rangle.
  \end{align*}
  Using \eqref{fhest2}, \eqref{interpolphi}, Lemma~\ref{fhproperties}, \eqref{eq:regularity} as well as Lemma~\ref{lem:norm-equiv} we infer that for $\phi_h \in V_h^B$
  \begin{equation}
    \label{consist}
    \begin{aligned}
      \abs{ a_h(u^e,\phi_h) - l_h(\phi_h) }
      & \leq c \norm{ \nabla u^e }_{L^2(D_h)} \left( \int_{D_h} \abs{ \nabla \phi_h \cdot \nu }^2 \dd x \right)^{\frac{1}{2}} \\
      & \quad + c h \norm{ f^e }_{L^2(D_h)} \norm{ \phi_h }_{L^2(D_h)}
      + c h^2 \norm{ u }_{H^2(\Gamma)} \norm{ \phi_h }_h  \\
      & \leq c h \norm{ f }_{L^2(\Gamma)} \left( \frac{1}{2h} \int_{D_h} \abs{ \nabla \phi_h \cdot \nu }^2 dx \right)^{\frac{1}{2}}
      + ch^2 \norm{ f }_{L^2(\Gamma)} \norm{ \phi_h }_h,
    \end{aligned}
  \end{equation}
  so that  \eqref{err1} implies the following intermediate result:
  \begin{equation}
    \label{intermediate}
    \norm{ e_h }_h = \left( \frac{1}{2 h} \int_{D_h} \bigl( \abs{ \nabla e_h }^2 + e_h^2 \bigr) | \nabla I_h \Phi | \dd x \right)^{\frac12}
    \leq  ch \norm{ f }_{L^2(\Gamma)}.
  \end{equation}

  In order to improve the $L^2$-error bound we define $\tilde{e}_h:= e_h \circ F_h^{-1}$ as well as
  \begin{equation*}
    \tilde{E}_h(p):= \frac{1}{2h} \int_{-h}^h \tilde{e}_h(F(p,s)) \dd s, \quad p \in \Gamma
  \end{equation*}
  with $F$ as above. We denote by $w \in H^2(\Gamma)$ the unique solution of
  \begin{equation*}
    - \Delta_{\Gamma} w + w = \tilde{E}_h \quad \mbox{ on } \Gamma,
  \end{equation*}
  which satisfies
  \begin{equation}
    \label{eq:dual-reg1}
    \norm{ w }_{H^2(\Gamma)} \le c \norm{ \tilde{E}_h }_{L^2(\Gamma)}.
  \end{equation}
  Similar to \eqref{extendedeq} the extension $w^e$ solves
  \begin{equation*}
    - \frac{1}{| \nabla \Phi | } \nabla \cdot \bigl( | \nabla \Phi | \, \nabla w^e \bigr) + w^e = \tilde{E}_h^e + \tilde{R} \quad \mbox{ in } U_{\delta},
  \end{equation*}
  where $\tilde{R}$ is obtained from \eqref{rest} by replacing $u$ by $w$. Using the transformation rule together with \eqref{transf1} we obtain
  \begin{align*}
    \norm{ \tilde{E}_h }_{L^2(\Gamma)}^2
    & = \frac{1}{2h} \int_{-h}^h \int_{\Gamma} \tilde{E}_h \tilde{e}_h \circ F \dd \sigma_p \dd s
    = \frac{1}{2h} \int_{-h}^h \int_{\Gamma} \tilde{E}^e_h \circ F \, \tilde{e}_h \circ F \, | \nabla \Phi \circ F | \, \mu \dd \sigma_p \dd s \\
    & \quad + \frac{1}{2h} \int_{-h}^h \int_{\Gamma} \tilde{E}_h  \, \tilde{e}_h \circ F \, (1- | \nabla \Phi \circ F| \, \mu ) \dd \sigma_p \dd s \\
    & = \frac{1}{2h} \int_{D^h} \tilde{E}^e_h \, e_h \circ F_h^{-1} | \nabla \Phi | \,  \dd x +  \frac{1}{2h} \int_{-h}^h \int_{\Gamma} \tilde{E}_h  \, \tilde{e}_h \circ F \, (1- | \nabla \Phi \circ F| \, \mu) \dd \sigma_p \dd s.
  \end{align*}
  The first term can be rewritten with the help of Lemma~\ref{aux} (applied to $w$ instead of $u$) to give
  \begin{align*}
    \norm{ \tilde{E}_h }_{L^2(\Gamma)}^2
    & = a_h(w^e,e_h) - \langle \tilde{S}, e_h \rangle
    - \frac{1}{2h} \int_{D_h} ( \nabla w^e \cdot \nabla \eta_h)  ( \nabla e_h \cdot \nu) \frac{| \nabla I_h \Phi |}{ | \nabla \Phi |} \dd x \\
    & \qquad + \frac{1}{2h} \int_{-h}^h \int_{\Gamma} \tilde{E}_h  \, \tilde{e}_h \circ F \, (1- | \nabla \Phi \circ F | \, \mu) \dd \sigma_p \dd s \\
    & \equiv \sum_{k=1}^4 I_k.
  \end{align*}
  In view of Lemma~\ref{interpol2}, \eqref{consist}, the fact that
  $\nabla w^e \cdot \nu =0$ and \eqref{intermediate} we have
  \begin{align*}
    & | I_1 | + | I_2 |  \leq  | a_h(w^e-I_h w^e,e_h) | + | a_h(u^e,I_h w^e) - l_h(I_h w^e) | + | \langle \tilde{S},e_h \rangle | \\
    & \quad \leq c h \norm{ w }_{H^2(\Gamma)} \norm{ e_h }_h + ch^2 \norm{ f }_{L^2(\Gamma)} \norm{ I_h w^e }_h  \\
    & \qquad + ch \norm{ f }_{L^2(\Gamma)} \bigl( \frac{1}{2h} \int_{D_h} | \nabla (I_h w^e - w^e) \cdot \nu |^2 dx
    \bigr)^{\frac{1}{2}} + c h^2 \norm{ w }_{H^2(\Gamma)} \norm{ e_h }_h  \\
    & \quad \leq c h^2 \norm{ f }_{L^2(\Gamma)} \norm{ w }_{H^2(\Gamma)}.
  \end{align*}
  Furthermore, \eqref{interpolphi}, \eqref{transf1} and \eqref{intermediate} imply
  \begin{equation*}
    | I_3 | + | I_4 |  \leq  c h \left( \norm{ w^e }_h + \norm{ \tilde{E}_h }_{L^2(\Gamma)} \right)  \norm{ e_h }_h
    \leq  c h^2 \norm{ f }_{L^2(\Gamma)} \left( \norm{ w }_{H^2(\Gamma)} + \norm{ \tilde{E}_h }_{L^2(\Gamma)} \right),
  \end{equation*}
  so that we obtain together with \eqref{eq:dual-reg1}
  \begin{equation*}
    \norm{ \tilde{E}_h }_{L^2(\Gamma)} \leq c h^2 \norm{ f }_{L^2(\Gamma)}.
  \end{equation*}
  Next, since $F(p,0)=p$ we may write for $p \in \Gamma$
  \begin{equation*}
    \tilde{E}_h(p) - \tilde{e}_h(p)
    =  \frac{1}{2h} \int_{-h}^h \int_0^s \nabla \tilde{e}_h(F(p,\tau)) \cdot \frac{\partial F}{ \partial s}(p, \tau) \dd \tau \dd s,
  \end{equation*}
  and hence we obtain with the help of \eqref{intermediate}
  \begin{equation*}
    \norm{ \tilde{E}_h - \tilde{e}_h }_{L^2(\Gamma)} \leq c \sqrt{h}  \left( \int_{D^h} | \nabla \tilde{e}_h |^2 \dd x \right)^{\frac12} \leq  c h \norm{ \nabla e_h }_h \leq c h^2 \norm{ f }_{L^2(\Gamma)}.
  \end{equation*}
  In conclusion we deduce that
  \begin{equation*}
    \norm{ e_h }_{L^2(\Gamma_h)} \leq c \norm{ \tilde{e}_h }_{L^2(\Gamma)} \leq \norm{ \tilde{E}_h - \tilde{e}_h }_{L^2(\Gamma)} +
    \norm{ \tilde{E}_h }_{L^2(\Gamma)} \leq c h^2 \norm{ f }_{L^2(\Gamma)}
  \end{equation*}
  and the theorem is proved.
  \qquad
\end{proof}


\section{A hybrid method for equations on evolving surfaces}
\label{hybridmoving}

\subsection{The setting}

The aim of this section is to combine ideas employed in \S\ref{SIF} and \S\ref{NBM} for the stationary problem in order to develop a finite element method for an advection--diffusion equation on a familiy of evolving hypersurfaces. More precisely, let $(\Gamma(t))_{t \in [0,T]}$ be a family of compact, connected smooth hypersurfaces embedded in $\R^{n+1}$ for $n = 1,2$. We suppose that
\begin{equation*}
  \Gamma(t) = \lbrace x \in \N(t) \, | \, \Phi(x,t) = 0 \rbrace, \quad \mbox{ where } \nabla \Phi(x,t) \neq 0, x \in \N(t)
\end{equation*}
and $\N(t)$ is an open neighbourhood of $\Gamma(t)$. We assume that $\N(t)$ is chosen so small that we can construct the function $p(\cdot,t)$ as in \S\ref{surfcalc}.

Given a velocity field $v(\cdot,t):\Gamma(t) \rightarrow \mathbb{R}^{n+1}$ we then consider the following initial value problem
\begin{subequations}
  \begin{align}
    \label{eq:adv-diff}
    \md u + u \nabla_\Gamma \cdot v - \Delta_\Gamma u  & = f \quad \mbox{ on } \bigcup_{t \in (0,T)} \Gamma(t) \times \{ t \}, \\
    \label{eq:initial}
    u(\cdot, 0)  & = u_0 \quad \mbox{ on } \Gamma(0).
  \end{align}
\end{subequations}
Here, $\md \eta$ denotes the material derivative of a function $\eta: \bigcup_{t \in (0,T)} \Gamma(t) \times \{ t \} \to \R$ which is given by
\begin{equation*}
  \md \eta = \partial_t \eta + v \cdot \nabla \eta,
\end{equation*}
if $\eta$ is extended into a neighbourhood of $\bigcup_{t \in (0,T)} \Gamma(t) \times \{ t \}$.

\subsection{The method}

In order to discretize the above problem we choose a partition  $0 =t_0 < t_1 < \ldots < t_N = T$ of $[0,T]$ with $\tau_m:=t_{m+1}-t_m,m=0,\ldots,N-1$ and $\tau:=\max_{m=0,\ldots,N-1} \tau_m$. Also, let $\T_h$ be an unfitted regular triangulation with mesh size $h$ of a region containing $\N(t), t \in [0,T]$. For $m=0,1,\ldots,N$ we set
\begin{align*}
  \Gamma_h^m & = \{ x \in \N(t_m) \, | \, I_h \Phi( x, t_m ) = 0 \} \\
  D_h^m & = \{ x \in \N(t_m) \, | \, \abs{ I_h \Phi( x, t_m ) } < h \},
\end{align*}
as well as
\begin{equation*}
  \T_h^m:= \lbrace T \in \T_h \, | \, \Hm^{n+1}(T \cap D_h^m) >0 \rbrace \quad \mbox{ and } \quad
  U_h^m := \bigcup_{T \in \T_h^m} T.
\end{equation*}
Here we assume that $0<h\leq h_0$, where $h_0$ is chosen so small that there exists $c_0, c_1 > 0$ such that
\begin{equation*}
  c_0 \leq \abs{ \nabla I_h \Phi(x,t) } \leq c_1, \quad (x,t) \in \bigcup_{t \in (0,T)} \N(t) \times \{ t \}.
\end{equation*}
Finally, we introduce
\begin{equation*}
  V_h^m = \{ \phi_h \in C^0(U_h^m) \, | \, \phi_h|_T \in P_1(T) \mbox{ for each } T \in \T_h^m \}.
\end{equation*}
In what follows we shall frequently use the abbreviation $z^m(x):=z(x,t_m)$.

In order to motivate our method we fix $m \in \lbrace 0,1,\ldots,N-1 \rbrace$ and let $\Psi$ be the solution of
\begin{equation*}
  \Psi_t(x,t) + D\Psi(x,t) v^e(x,t) = 0, \quad \Psi(x,t_{m+1})=x,
\end{equation*}
where $v^e(x,t):=v(p(x,t),t)$. For a sufficiently smooth function $\varphi:\N(t_{m+1}) \to \R$ we define $\eta(x,t):=\varphi(\Psi(x,t))$. Clearly, $\eta(\cdot,t_{m+1})=\varphi$ and a short calculation shows that $\md \eta = 0$. Assuming that $u$ is a solution of \eqref{eq:adv-diff} we obtain with the help of the  Leibniz formula and integration by parts
\begin{align*}
  & \frac{d}{dt} \int_{\Gamma(t)} u \eta \dd \sigma_{| t = t_{m+1}}
  = \int_{\Gamma(t_{m+1})} \bigl( \md (u \eta) + u \eta \nabla_{\Gamma} \cdot v \bigr) \dd \sigma   \\
  & \quad = \int_{\Gamma(t_{m+1})} \varphi \bigl( \md u +  u \nabla_{\Gamma}\cdot  v \bigr) \dd \sigma
  = \int_{\Gamma(t_{m+1})} \bigl( \varphi \Delta_{\Gamma} u + \varphi f \bigr) \dd \sigma \\
  & \quad = - \int_{\Gamma(t_{m+1})} \nabla_{\Gamma} u
  \cdot \nabla_{\Gamma} \varphi \dd \sigma + \int_{\Gamma(t_{m+1})} f \varphi \dd \sigma.
\end{align*}
Since $\Psi(\cdot,t_{m+1}) \equiv \mbox{id}$, a Taylor expansion shows that
\begin{align*}
  \Psi(x,t_m) & = \Psi(x,t_{m+1}-\tau_m) \approx x - \tau_m \Psi_t(x,t_{m+1}) \\
  & = x + \tau_m D\Psi(x,t_{m+1}) v^{e,m+1}(x) = x + \tau_m v^{e,m+1}(x).
\end{align*}
Thus we may approximate the left hand side of the above relation by
\begin{align*}
  \frac{d}{dt} \int_{\Gamma(t)} u \eta \dd \sigma_{|t = t_{m+1}}
  & \approx \frac{1}{\tau_m} \left\{ \int_{\Gamma(t_{m+1})} u^{m+1} \varphi \dd \sigma - \int_{\Gamma(t_m)} u^m \varphi(\cdot+ \tau_m v^{e,m+1}) \dd \sigma \right\}.
\end{align*}

The above calculations motivate the following scheme, in which we use the narrow band approach in order to discretize the elliptic part. Given  $u_h^m \in V^m_h, m=1,\ldots,N-1$, find $u^{m+1}_h \in V^{m+1}_h$ such that
\begin{equation}
  \label{eq:evolve-fem_alt}
  \begin{aligned}
    & \int_{\Gamma_h^{m+1}} u_h^{m+1} \phi_h \dd \sigma_h - \int_{\Gamma_h^{m}} u_h^m \phi_h(\cdot + \tau_m v^{e,m+1}) \dd \sigma_h \\
    & \quad + \frac{\tau_m}{2 h} \int_{D_h^{m+1}} \nabla u_h^{m+1} \cdot \nabla \phi_h \, | \nabla I_h \Phi^{m+1} |  \dd x = \tau_m \int_{\Gamma_h^{m+1}} f^{e,m+1} \phi_h \dd \sigma_h
  \end{aligned}
\end{equation}
for all $\phi_h \in V^{m+1}_h$. Here, $u^0_h = I_h u^0$. Existence and uniqueness of $u^{m+1}_h$ follows in a similar way as for the narrow band method in the elliptic case.

\subsection{Mass conservation}

An important property of solutions of \eqref{eq:adv-diff} is conservation of mass in the case that $\int_{\Gamma(t)} f(\cdot,t) \dd \sigma = 0$. The following lemma shows that our numerical scheme preserves this property under some mild constraints on the discretization parameters.

\begin{lemma}
  \label{lem:mass-conserv}
  Suppose that $ \int_{\Gamma_h^{m+1}} f^{e,m+1} \dd \sigma_h = 0, \, m = 0, \ldots, N-1$. Let $u^m_h \in V^m_h,m=1,\ldots,N$ be the solutions of \eqref{eq:evolve-fem_alt}. Then
  \begin{equation}
    \int_{\Gamma^m_h} u^m_h \dd \sigma_h = \int_{\Gamma^0_h} u^0_h \dd \sigma_h,
  \end{equation}
  provided that $0< h \leq h_1$ and $\tau \leq \gamma \sqrt{h}$.
\end{lemma}

\begin{proof}
  Let us first observe that
  \begin{equation}
    \label{gammah}
    \lbrace x + \tau_m v^{e,m+1}(x) \, | \, x \in \Gamma^m_h \rbrace  \subset U^{m+1}_h, \quad m=0,\ldots,N-1
  \end{equation}
  provided that $h, \tau$ are sufficiently small. To see this, let $x \in \Gamma^m_h$ and choose an element $T \in \T_h$ such that $x \in T$. Then,
  \begin{equation*}
    \Phi^{m+1}(x+ \tau_m v^{e,m+1}(x)) = \Phi^m(x) +  \tau_m \nabla \Phi^m(x) \cdot v^{e,m+1}(x) + \tau_m \Phi_t(x,t_m) + R_m(x),
  \end{equation*}
  where $\abs{ R_m(x) } \leq c \tau_m^2$. Observing that $\Phi_t + \nabla \Phi \cdot v = 0 $ on $\bigcup_{t \in (0,T)} \Gamma(t) \times \{ t \}$ we may write
  \begin{align*}
    & \nabla \Phi^m(x) \cdot v^{e,m+1}(x) + \Phi_t( x, t_m ) = \nabla \Phi^m(x) \cdot v^{m+1}(p^{m+1}(x)) + \Phi_t( x, t_m ) \\
    & \quad = \bigl( \nabla \Phi^m(x) - \nabla \Phi^{m+1}(p^{m+1}(x)) \bigr) \cdot  v^{m+1}(p^{m+1}(x)) + \Phi_t( x, t_m ) - \Phi_t(p^{m+1}(x),t_{m+1}),
  \end{align*}
  so that
  \begin{align*}
    & \abs{ \Phi^{m+1}(x+ \tau_m v^{e,m+1}(x)) } \leq \abs{ \Phi^m(x) } + c \tau_m  \abs{ x -p^{m+1}(x) } +  c  \tau_m^2 \\
    & \leq \abs{ \Phi^m(x) } + c \tau_m \abs{ \Phi^{m+1}(x) } + c  \tau_m^2  \leq c (h^2 + \tau_m^2),
  \end{align*}
  in view of \eqref{pk} and since $\abs{ \Phi^m(x) } = \abs{ \Phi^m(x) - I_h \Phi^m(x) } \leq c h^2$. As a result,
  \begin{equation*}
    \abs{ (I_h \Phi^{m+1})(x+\tau_m v^{e,m+1}(x)) } \leq \norm{ I_h \Phi^{m+1}- \Phi^{m+1} }_{L^{\infty}} + c (h^2 + \tau_m^2) \leq c (h^2 + \tau_m^2) < h,
  \end{equation*}
  provided that $0 < h \leq h_1$ and $\tau \leq \gamma \sqrt{h}$. Hence, $x + \tau_m v^{e,m+1}(x) \in D^{m+1}_h \subset U^{m+1}_h$ proving \eqref{gammah}. The result of the lemma now follows from inserting $\phi_h \equiv 1 \in V^{m+1}_h$ into \eqref{eq:evolve-fem_alt} and using \eqref{gammah} together with our assumption that $\int_{\Gamma_h^{m+1}} f^{e,m+1} \dd \sigma_h = 0$.
  \qquad
\end{proof}


\section{Numerical Experiments}
\label{numerical}

\subsection{Notes on implementation}

The methods were implemented using the Distributed and Unified Numerics Environment (\textsf{DUNE}) \cite{BasBlaDed08-a,BasBlaDed08-b,DedKloNol10}. Assembly of the matrices is non standard in that the method requires integration over partial elements. To do so we subdivide the integration areas in simplices using the \textsf{Triangle} \cite{She96,triangle} and \textsf{Tetgen} \cite{Si06} packages. In each case, the linear system is solved with the conjugate gradient method until the residual is reduced by a factor of $10^{-8}$ in comparison to its initial value in the $\ell^2$ norm. Due to the lack of shape regularity of $\Gamma_h$ and $D_h$, the matrix systems are ill conditioned and so we used a Jacobi preconditioner in order to speed up the convergence of our iterative solver. In practice, we will take $U_h$ to be a subset of a cube shaped domain. The triangulation $\T_h$ will be computed adaptively refining only those elements which intersect the computational domain, either $\Gamma_h$ or $D_h$. Given errors $E_i$ and $E_{i-1}$ at two different mesh sizes $h_i$ and $h_{i-1}$, we calculate the experimental order of convergence (eoc) by
\begin{equation}
  \label{eq:eoc}
  \mathrm{(eoc)}_i = \frac{ \log( E_i / E_{i-1} ) }{ \log( h_i / h_{i-1} ) }.
\end{equation}

\subsection{Poisson equation}

To test our methods, we present two numerical examples. The first is on a torus and is taken from \cite{OlsReuGra09} and the second is on a potato-like surface from \cite{Dzi88}.

We define the torus through the signed distance function:
\begin{equation*}
  \Gamma = \{ x \in \R^3 \, | \, d(x) = 0 \}
  \quad
  d(x) = \sqrt{ \left( \sqrt{ x_1^2 + x_2^2 } - R \right)^2 + x_3^2 } - r,
\end{equation*}
for $R = 1, r = 0.6$. To compute our exact solution, we parameterise the torus by
\begin{equation*}
  x_1 = ( R + r \cos \theta ) \cos \varphi, \quad x_2 = ( R + r \cos \theta ) \sin \varphi, \quad x_3 = r \sin \theta, \quad \mbox{ for } \theta, \varphi \in (-\pi,\pi)
\end{equation*}
and take the exact solution
\begin{equation*}
  u( \theta, \varphi ) = \cos( 3 \varphi ) \sin( 3 \theta + \varphi ).
\end{equation*}
For this example, \eqref{eq:Finv} can be calculated analytically.

For our second example, we set $\Gamma = \{ x \in \R^3 \, | \, \Phi(x) = 0 \}$ using the level set function:
\begin{equation*}
  \Phi(x) = ( x_1 - x_3^2 )^2 + x_2^2 - 1.
\end{equation*}
From $\Phi$, we calculate the normal $\nu = \frac{\nabla \Phi}{\abs{\nabla \Phi}}$ and the mean curvature by
\begin{equation*}
  H = \nabla \cdot \frac{\nabla \Phi}{\abs{\nabla \Phi}} = \frac{1}{\abs{\nabla \Phi}} \sum_{j,k=1}^3 \left( \delta_{jk} - \frac{ \Phi_{x_j} \Phi_{x_k} }{ \abs{ \nabla \Phi }^2 } \right) \Phi_{x_j x_k}.
\end{equation*}
As exact solution, we take $u(x) = x_1 x_2$ and calculate the right-hand side $f$ as $f = -\Delta_\Gamma u + u$ as
\begin{equation*}
  f(x) = 2 \nu_1(x) \nu_2(x) + H(x) ( x_2 \nu_1(x) + x_1 \nu_2(x) ), \quad x \in \Gamma.
\end{equation*}
For this example, \eqref{eq:Finv} can not be calculated exactly so we approximate using a gradient decent like iteration from \cite{Reu13} originally for the closest point operator.

The errors $\norm{ u^e - u_h }_{L^2(\Gamma_h)}$ for \emph{SIF} and \emph{NBM} are shown in Tables~\ref{tab:poisson-torus} and \ref{tab:poisson-dziuk}. The numerical results confirm the theoretical bounds from Theorems~\ref{thm:line-bound} and \ref{thm:band-bound}. Results are also available for the $H^1$-semi-norm error.

\begin{table}
  \footnotesize
  \centering

  \begin{tabular}{ccccc}
    \hline
    $h$ &  $\norm{ u^e - u_h^{\mbox{\emph{SIF}}} }_{L^2(\Gamma_h)}$ & (eoc) &  %
    $\norm{ u^e - u_h^{\mbox{\emph{NBM}}} }_{L^2(\Gamma_h)}$ &  (eoc) \\
    \hline
$2^{-1}\sqrt{3}$ & $1.67739$ & --- & $1.61943$ & --- \\
$2^{-2}\sqrt{3}$ & $7.10825 \cdot 10^{-1}$ & $1.238650$ & $7.07220 \cdot 10^{-1}$ & $1.195260$ \\
$2^{-3}\sqrt{3}$ & $1.90004 \cdot 10^{-1}$ & $1.903470$ & $2.32053 \cdot 10^{-1}$ & $1.607700$ \\
$2^{-4}\sqrt{3}$ & $4.73865 \cdot 10^{-2}$ & $2.003480$ & $7.17605 \cdot 10^{-2}$ & $1.693190$ \\
$2^{-5}\sqrt{3}$ & $1.19721 \cdot 10^{-2}$ & $1.984800$ & $1.97350 \cdot 10^{-2}$ & $1.862430$ \\
$2^{-6}\sqrt{3}$ & $3.01376 \cdot 10^{-3}$ & $1.990040$ & $5.08158 \cdot 10^{-3}$ & $1.957410$ \\
    \hline
  \end{tabular}

  \caption{Error tables of sharp interface method (\emph{SIF}) and narrow band method (\emph{NBM}) for the first test problem. These calculations took successively 38, 69, 128, 240, 359, 641 conjugate gradient iterations for \emph{SIF} and 33, 54, 97, 182, 392, 634 conjugate gradient iterations for \emph{NBM}.}
  \label{tab:poisson-torus}
\end{table}

\begin{table}
  \footnotesize
  \centering

  \begin{tabular}{ccccc}
    \hline
    $h$ &  $\norm{ u^e - u_h^{\mbox{\emph{SIF}}} }_{L^2(\Gamma_h)}$ & (eoc) &  %
    $\norm{ u^e - u_h^{\mbox{\emph{NBM}}} }_{L^2(\Gamma_h)}$ &  (eoc) \\
    \hline
$2^{-1}\sqrt{3}$ & $3.31237 \cdot 10^{-1}$ & --- & $2.31128 \cdot 10^{-1}$ & --- \\
$2^{-2}\sqrt{3}$ & $9.97842 \cdot 10^{-2}$ & $1.730980$ & $9.06054 \cdot 10^{-2}$ & $1.351020$ \\
$2^{-3}\sqrt{3}$ & $2.57329 \cdot 10^{-2}$ & $1.955200$ & $2.57213 \cdot 10^{-2}$ & $1.816630$ \\
$2^{-4}\sqrt{3}$ & $6.59538 \cdot 10^{-3}$ & $1.964090$ & $7.43214 \cdot 10^{-3}$ & $1.791110$ \\
$2^{-5}\sqrt{3}$ & $1.64586 \cdot 10^{-3}$ & $2.002610$ & $1.94710 \cdot 10^{-3}$ & $1.932450$ \\
$2^{-6}\sqrt{3}$ & $4.10269 \cdot 10^{-4}$ & $2.004200$ & $4.99422 \cdot 10^{-4}$ & $1.962990$ \\
$2^{-7}\sqrt{3}$ & $1.02735 \cdot 10^{-4}$ & $1.997640$ & $1.26086 \cdot 10^{-4}$ & $1.985850$ \\
    \hline
  \end{tabular}

  \caption{Error tables of sharp interface method (\emph{SIF}) and narrow band method (\emph{NBM}) for the second test problem. These calculations took successively 25, 53, 103, 196, 298, 585, 1151 conjugate gradient iterations for \emph{SIF} and 37, 68, 116, 153, 297, 580, 1137 conjugate gradient iterations for \emph{NBM}.}
  \label{tab:poisson-dziuk}
\end{table}

To compare with other methods, we also include a plot of the error in the $L^2$-norm against $h$ for \emph{SIF} and \emph{NBM} along with the unfitted finite element methods of \cite{OlsReuGra09,DecDziEll10}. The plot shows that the error on $\Gamma_h$ is almost the same for each of the four methods considered.

\begin{figure}
  \centering
  \includegraphics{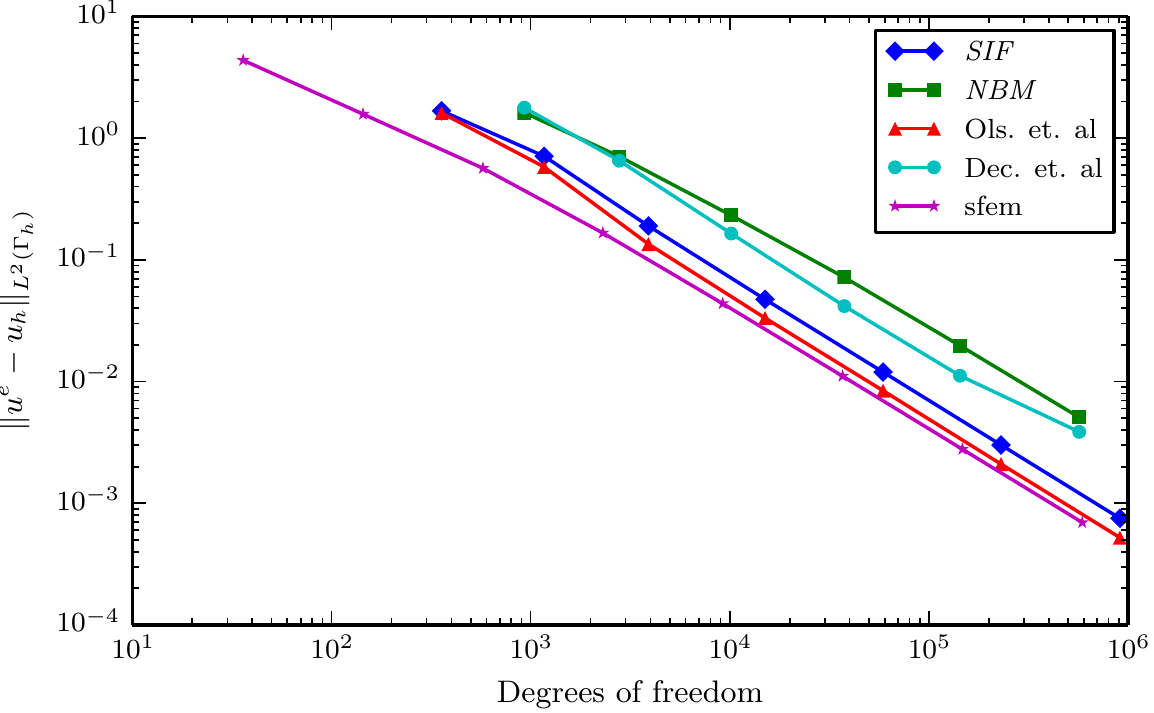}
  \caption{Plot of the $L^2$ error of various methods for the first test problem.}

  \centering
  \includegraphics{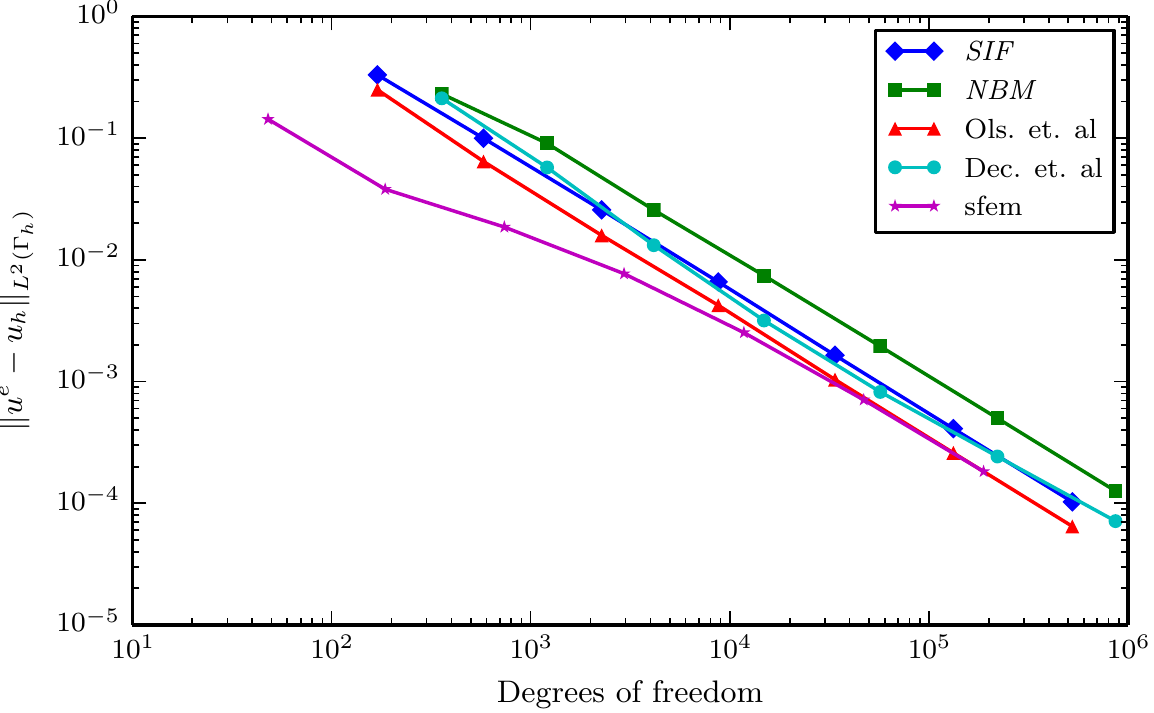}
  \caption{Plot of the $L^2$ error of various methods for the second test problem.}
\end{figure}

\begin{figure}
  \centering
  \includegraphics[width=0.45\textwidth]{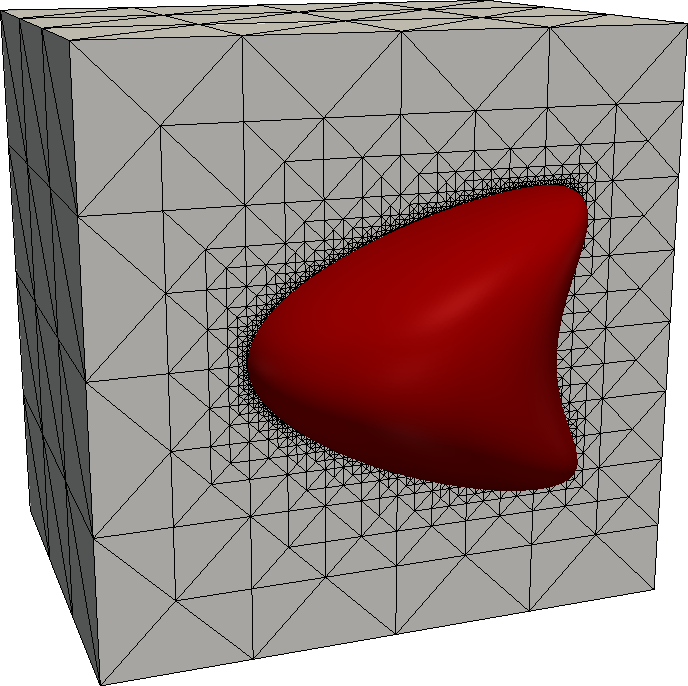}
  \includegraphics[width=0.4\textwidth]{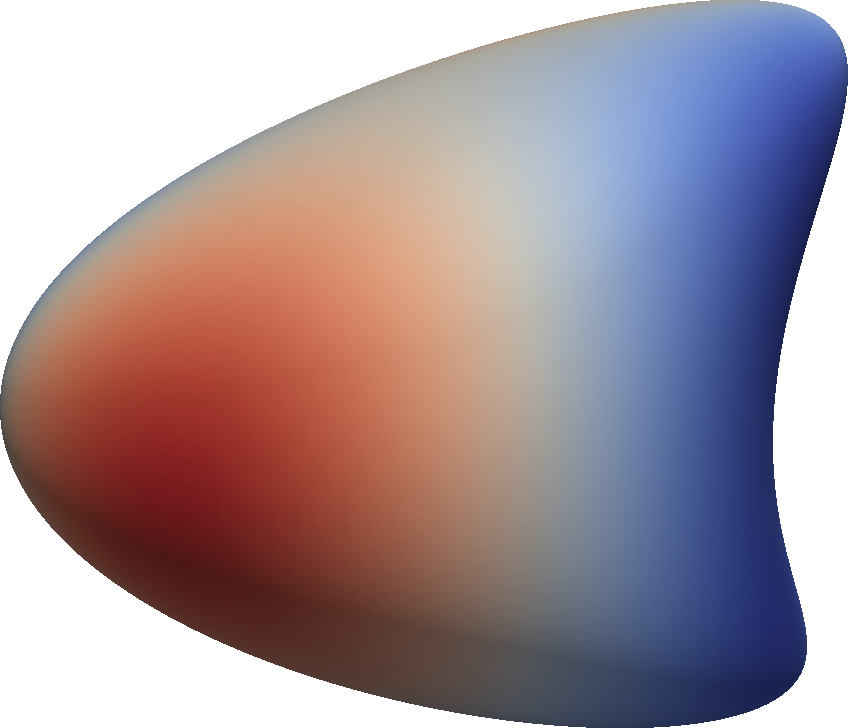}

  \caption{Plots of the computation domain (left) and solution (right) with $h = 2^{-6} \sqrt{3}$ for the second problem.}
\end{figure}

Further numerical examples are available in \cite{Ran13}.

\subsection{Parabolic equation on an evolving curve}

For an example of an evolving curve we take $\Gamma(t) = \{ x \in \R^2 \, | \, \Phi(x,t) = 0 \}$ for
\begin{equation*}
  \Phi(x,t) = \frac{x^2}{ 1 + \frac{1}{4} \sin( 2 \pi t ) } + y^2 - 1,
\end{equation*}
for $t \in [0,\tfrac12]$. We calculate a right-hand side $f$ so that the exact solution is $u(x,t) = \exp( - 4 t ) x_1 x_2$. Taking $\tau = 2 h^2$, the scheme demonstrates second order convergence in the $L^2(\Gamma_h^m)$-norm; see Table~\ref{tab:hybrid-2d}. Numerical experiments confirm the conservation of mass result.

\begin{table}
  \footnotesize
  \centering

  \begin{tabular}{ccc}
    \hline
    $h$ & $\max_m \| u(t^m) - u_h^m \|_{L^2(\Gamma_h^m)}$ & (eoc) \\
    \hline
    $2^{-1} \sqrt{2}$ & $ 1.15457 \cdot 10^{-1} $ & ---  \\
    $2^{-2} \sqrt{2}$ & $ 3.25344 \cdot 10^{-2} $ & $ 1.82732 $ \\
    $2^{-3} \sqrt{2}$ & $ 8.64172 \cdot 10^{-3} $ & $ 1.91258 $ \\
    $2^{-4} \sqrt{2}$ & $ 2.13241 \cdot 10^{-3} $ & $ 2.01883 $ \\
    $2^{-5} \sqrt{2}$ & $ 5.42960 \cdot 10^{-4} $ & $ 1.97357 $ \\
    \hline
  \end{tabular}

  \caption{Results of the hybrid scheme for a parabolic equation on an evolving curve.}
  \label{tab:hybrid-2d}
\end{table}


\end{document}